\documentclass[11pt]{article}

\usepackage{amsthm}
\usepackage{amsmath}
\usepackage{amsfonts, epsfig, amsmath, amssymb, color, amscd}
\usepackage[cmtip,arrow,all]{xy}
\usepackage{pb-diagram, pb-xy}
\usepackage{amssymb,epsfig}
\usepackage{color}
\usepackage[cmtip,arrow]{xy}
\usepackage{pb-diagram, pb-xy}
\usepackage{amssymb,epsfig,amsfonts}

\usepackage[linktocpage=true,plainpages=false,pdfpagelabels=false]{hyperref}

\usepackage{tikz}
\usepackage{eufrak}

\newfont{\sdbl}{msbm9}
\newfont{\dbl}{msbm10 at 12pt}
\theoremstyle{definition}

\newcommand{\dz}{{\mbox{\dbl Z}}}

\newcommand{\dn}{{\mbox{\dbl N}}}

\newcommand{\sdz}{{\mbox{\sdbl Z}}}

\newcommand{\dq}{{\mbox{\dbl Q}}}

\newcommand{\ord}{\mathop{\rm ord}\nolimits}

\newcommand{\idm}{\mathfrak{m}}
\newcommand{\Ord}{\mathop {\rm \bf ord}}
\newcommand{\bd}{\mathbb{D}}
\newcommand{\bl}{\mathbb{L}}

\makeatother
\newtheorem{Def}{Definition}[section]
\newtheorem{rem}{Remark}[section]

\newtheorem{ex}{Example}[section]

\theoremstyle{plain}

\newtheorem{theorem}{Theorem}[section]
\newtheorem{lemma}{Lemma}[section]
\newtheorem{cor}{Corollary}[section]

\numberwithin{equation}{section}

\newcommand{\crr}{{{\cal R}}}
\newcommand{\cf}{{{\cal F}}}

\newcommand{\cm}{{{\cal M}}}

\begin{document} 
\title{Normal forms for ordinary differential operators, II}
\author{J. Guo \and  A.B. Zheglov}
\date{}
\maketitle


\begin{abstract}
In this paper, which is  a follow-up of our first paper "Normal forms for ordinary differential operators, I", we extend the theory of normal forms for non-commuting operators, and obtain as an  application  a commutativity criterion for operators in the Weyl algebra or, more generally, in the ring of ordinary differential operators, which we prove in the case when operators have a normal form with  the  restriction top line (for details see Introduction).

\end{abstract}


\markright{Normal forms for ODOs}

\tableofcontents

\section{Introduction}
\label{S:introduction}

This paper is a follow-up of \cite{GZ24} and we use its notation. We collect all necessary notation in the list \ref{S:list} below, for details we refer to  \cite{GZ24}.

In \cite{GZ24} we developed the generic theory of normal forms for ordinary differential operators, which was conceived and developed as a part of the generalised Schur theory offered in \cite{A.Z}, and applied it to obtain  a new explicit parametrisation of torsion free rank one sheaves on projective irreducible curves with vanishing cohomology groups.

In this paper we obtain the second application -- a commutativity criterion for operators in the Weyl algebra or, more generally, in the ring of ordinary differential operators. It is motivated by the following natural question from the Burchnall-Chaundy theory.  

The famous Burchnall-Chaundy lemma (\cite{BC1}) says  that any two commuting differential operators $P,Q\in D_1:=K[[x]][\partial ]$ are algebraically dependent. More precisely, if the orders $n,m$ of operators $P,Q$ are coprime\footnote{i.e. the rank of the ring $K[P,Q]$ is 1, see e.g. \cite{Zheglov_book} for relevant definitions, in particular \cite[Lemma 5.23]{Zheglov_book} for a proof of the Burchnall-Chaundy lemma in general form. The statement about the form of polynomial follows easily from the proof.}, then there exists an irreducible polynomial $f(X,Y)$ of weighted degree $v_{n,m}(f)=mn$ of {\it special form} (here the weighted degree is defined as in Dixmier's paper \cite{Dixmier}, cf. item 4 in the List of Notations below): $f(X,Y) = \alpha X^m\pm Y^n+\ldots $ (here $\ldots$ mean terms of lower weighted degree, $0\neq\alpha\in K$; in particular, for coprime $n$ and $m$ the polynomial $f$ is automatically irreducible), such that $f(P,Q)=0$. A similar result for commuting operators of rank $r$ was established in \cite{Wilson} (cf. \cite[Th. 2.11]{Previato2019}), in this case $m=\ord (Q)/r$, $n=\ord (P)/r$, and again $GCD (m,n)=1$. Vice versa, if $P,Q\in D_1$ is a solution of such polynomial $f(X,Y)$\footnote{A solution of the equation $f(X,Y)=0$  is a pair $(P,Q)\in D$ such that $\sum_{i,j=0}^n \alpha_{ij}P^iQ^j = 0$.}, then  $[P,Q]=0$. Now a natural question whether $F(P,Q)=0 \Rightarrow [P,Q]=0$ for generic polynomial $F$ appears. This question appears to be surprisingly difficult in general case. We give a partial affirmative answer on this question in the case when the normal form has the {\it restriction top line} (see discussion below).

Recall that {\it a normal form} of a pair of operators $P,Q\in \hat{D}_1^{sym}$ is a pair $P', Q' \in  \hat{D}_1^{sym}$ obtained after conjugation by a 
Schur operator $S$ as above, calculated using one of the operators in a pair $(P,Q)$ (or, more generally, in the ring $\hat{D}_1^{sym}$, see definitions 2.7, 3.6 and remark 2.7 in \cite{GZ24}). The normal form is not uniquely defined, but up to conjugation with invertible $S\in \hat{D}_1^{sym}$ from the centralizer $C(\partial^k)$  with $\Ord (S)=0$. Such $S$ is known to be a {\it polynomial} of restricted degree. Notably, the whole centraliser $C(\partial^k)$ is naturally isomorphic to a matrix $k\times k$ algebra over a polynomial ring, see remark 3.3 in \cite{GZ24}.
The normal form of any pair of {\it commuting operators} can be explicitly calculated. If the operators do not commute, the normal forms will  be series in general, for which, however, it is possible to calculate any given number of terms. For a pair of {\it differential operators} normal forms satisfy condition $A_q(0)$.

To study normal forms of non-commuting operators we develop a technique of Newton regions (see section \ref{S:normal_forsm_for_non-commuting}) -- this is a natural generalisation of the technique of Newton polygons widely used for study of operators in the Weyl algebra (cf. \cite{Dixmier}, \cite{GGV}, \cite{GGV2}, \cite{ML}, \cite{ML2}). Since normal forms of non-commuting operators are usually infinite series, the convex hull of all monomials may not be  a restricted domain. However, in this case it is possible to define relevant notions of weights and top lines (generalisations of corresponding notions from \cite{Dixmier}). In section \ref{S:normal_forsm_for_non-commuting} we study normal forms of a pair of non-commuting monic differential operators $P,Q\in D_1$. After conjugating this pair by a Schur operator of, say, operator $Q$, we obtain a monic operator $P'\in \hat{D}_1^{sym}$ satisfying condition $A_q(0)$ (where $q$ is the order of $Q$). It is possible to define a weight function $v_{\sigma ,\rho}$ and a notion of related top line for such operators. We distinguish 2 principal cases of top lines: the restriction top line and the asymptotic  top line, both lines are uniquely defined (see definitions \ref{D:restriction top line} and \ref{D:asymptotic top line}). Lemma \ref{L: res or asym} says that there are only two possibilities for a non commuting with $\partial^q$ operator $P'$: it has either the restriction top line or the asymptotic top line. In section \ref{S:restriction top line} we give the affirmative answer on the question whether $F(P,Q)=0 \Rightarrow [P,Q]=0$ in the case when the normal form $P'$ of the pair $P,Q$ has the  restriction top line.

We will consider the remaining case of the asymptotic top line  in the next articles, since this case requires much more details. We hope that further development of the technique of normal forms and related concepts, which are touched upon in this work, will also allow us to approach the solution of other problems related to Weyl algebras, cf. e.g. the works \cite{Dixmier}, \cite{BEE}, \cite{BK1}, \cite{Bavula3}, \cite{Joseph}, \cite{Ts2}.

\bigskip

The structure of this article is the following. 

In section \ref{S:normal_forsm_for_non-commuting} we study normal forms of non-commuting differential operators. In section \ref{S:General_Newton} we introduce the notion of Newton region -- a natural generalisation of the Newton polygon -- for operators from $\hat{D}_1^{sym}$ and study its basic properties for operators satisfying condition $A_k(0)$ (all normal forms of differential operators satisfy this condition). In section \ref{S:Key combinatorial lemma} we prove one general combinatorial lemma, and in section \ref{S:restriction top line} we prove the main theorem of section \ref{S:normal_forsm_for_non-commuting} -- a commutativity criterion of a pair of operators in the case when the normal form of this pair has the restriction top line.  

In section \ref{S:appendix} we collect all necessary basic technical assertions about the weight function $v_{\sigma,\rho}$ and the homogeneous highest terms $ f_{\sigma,\rho}$ used in section \ref{S:normal_forsm_for_non-commuting}, with detailed proofs.

{\bf Acknowledgements.}  The authors' research was supported by the Moscow Center of Fundamental and Applied Mathematics of Lomonosov Moscow State University under agreement No. 075-15-2025-345. 

The work of J. Guo was also partially supported by the Leshan Normal University Scientific Research Start-up Project for Introducing High-level Talents.

We are grateful to the Sino-Russian Mathematics Center at Peking University for hospitality and excellent working conditions while preparing this paper.

We are grateful to Huijun Fan for his interest in our work, and for his advice and support to J. Guo. 

We are also grateful to the anonymous referees whose remarks allowed to improve the exposition.

\subsection{List of notations}
\label{S:list}

 Here we recall  the most important notations used in this paper from \cite{GZ24}.
 
1.  $\dz_+$ is the set of all non-negative integers, $\dn$ is the set of natural numbers (all positive integers). $K$ is a field of characteristic zero. Recall some notation from \cite{A.Z}:
	$\hat{R}:=K [[x_1,\ldots ,x_n]]$, the $K$-vector space 
$$
\cm_n := \hat{R} [[\partial_1, \dots, \partial_n]] = \left\{
\sum\limits_{\underline{k} \ge \underline{0}} a_{\underline{k}} \underline{\partial}^{\underline{k}} \; \left|\;  a_{\underline{k}} \in \hat{R} \right. \;\mbox{for all}\;  \underline{k} \in \dn_0^n
\right\},
$$
$\upsilon:\hat{R}\rightarrow \dn_0\cup \infty$ --  the discrete valuation defined by the unique maximal ideal $\idm = (x_1, \dots, x_n)$ of $\hat{R}$,  \\
for any element
$
0\neq P := \sum\limits_{\underline{k} \ge \underline{0}} a_{\underline{k}} \underline{\partial}^{\underline{k}} \in \cm_n
$
$$
\Ord (P) := \sup\bigl\{|\underline{k}| - \upsilon(a_{\underline{k}}) \; \big|\; \underline{k} \in \dn_0^n \bigr\} \in \dz \cup \{\infty \},
$$
$$
\hat{D}_n^{sym}:=\bigl\{Q \in \cm_n \,\big|\, \Ord (Q) < \infty \bigr\};
$$
$
P_m:= \sum\limits_{ |\underline{i}| - |\underline{k}| = m} \alpha_{\underline{k}, \underline{i}} \,  \underline{x}^{\underline{i}} \underline{\partial}^{\underline{k}}
$ -- the $m$-th \emph{homogeneous component} of $P$,\\
$\sigma (P):=P_{\Ord (P)}$ -- the highest symbol.
	
2. In this paper we use: 	
	$\hat{R}:=K[[x]]$, $D_1:=\hat{R}[\partial]$, 
$$\hat{D}_1^{sym}:=\{Q=\sum_{k\ge 0}a_k\partial^k|\Ord(Q)<\infty\}.$$ 

Operators:  $\delta:=\exp((-x)\ast \partial)$\footnote{ Here and further $\ast$ in all exponentials means that we consider normalized Taylor power series, i.e. the powers of $x$ always stand on the left of powers of $\partial$, for example $\delta:=\exp((-x)\ast \partial)=\sum_{k=0}^{+\infty}\frac{(-1)^k}{k!}x^k\partial^k$.},   $\int:=(1-exp((-x)\ast \partial))\partial^{-1} $,  	$A_{k;i}:=\exp((\xi^{i}-1)x\ast \partial)\in \hat{D}^{sym}_{1}\hat{\otimes}_{K}\tilde{K}$ (in the case when $k$ is fixed, simply written as $A_i$), where $\tilde{K}=K[\xi]$, $\xi$ is a primitive $k$th root of unity, 
$\Gamma_i=(x\partial)^i$. $B_n=\frac{1}{(n-1)!}x^{n-1}\delta\partial^{n-1}$. 

$\hat{D}^{sym}_{1}\hat{\otimes}_{K}\tilde{K}$ means the same ring $\hat{D}^{sym}_{1}$, but defined over the base field $\tilde{K}$.

The operator $P\in D_1$ is called {\it normalized} if $P=\partial^p+a_{p-2}\partial^{p-2}+\ldots $. The operator $P\in D_1$ is {\it monic} if its highest coefficient is 1. Analogously, $P\in \hat{D}_1^{sym}$ is monic if $\sigma (P)=\partial^p$. 

	We denote $D^i=\partial^i$ if $i\ge 0$ and $\int^{-i}$ if $i<0$. Operators written in the (Standard) form as 
	$$
	H=[\sum_{0\leq i<k}f_{i;r}(x, A_{k;i}, \partial )+\sum_{0<j\leq N}g_{j;r}B_{j}]D^{r}
	$$
	are called HCP and  form a sub-ring $Hcpc(k)$.   Here $f_{i;r}(x, A_{k;i}, \partial)$ is a polynomial of $x, A_{k;i},\partial$,  $\Ord(f_{i;r})=0$,  of the form
		$$
		f_{i;r}(x, A_{k;i}, \partial )=\sum_{0\leq l\leq d_{i}}f_{l,i;r}x^l A_{k;i}\partial^l
		$$
		for some $d_i\in \dz_+$, where $f_{l,i;r}\in \tilde{K}$. The number  $d_i$ is called the {\it $x$-degree of $f_{i;r}$}:  $deg_{x}(f_{i;r}):= d_i$; $g_{j;r}\in \tilde{K}$, $g_{j;r}=0$ for $j\le -r$ if $r<0$.

	They can be written also in G-form: 
	$$
	H=(\sum_{0\leq i<k}\sum_{0\leq l\leq d_i} f'_{l,i;r}\Gamma_lA_i+\sum_{0<j\leq N}g_{j;r}B_{j})D^{r}
	$$
	The $A$ and $B$ Stable degrees of HCP are defined as 
	$$
	Sdeg_A(H)=\max \{d_i|\quad 0\leq i<k \} \quad \mbox{or $-\infty$, if all $f_{l,i;r}=0$ } 
	$$
	and 
	$$Sdeg_B(H)=\max\{j|\quad g_{j;r}\neq0\} \quad \mbox{or $-\infty$, if all $g_{j;r}=0$}
	$$
	
	In the case when $Sdeg_B(HD^p)=-\infty,\forall p\in \mathbf{Z}$ $H$ is called {\it totally free of} $B_j$.
	
	An operator $P\in  \hat{D}_1^{sym}$ satisfies {\it condition $A_q(k)$}, $q,k\in \dz_+$, $q>1$ if 
	\begin{enumerate}
		\item
		$P_{t}$ is a HCP  from $Hcpc (q)$  for all $t$;
		\item
		$P_{t}$ is totally free of $B_j$ for all $t$;
		\item
		$Sdeg_A(P_{\Ord (P)-i})< i+k$ for all $i>0$;
		\item
		$\sigma (P)$ does not contain $A_{q;i}$, $Sdeg_A(\sigma (P))=k$.
	\end{enumerate}

3. In section 3, $\mathfrak{B}=\crr_S$ is the right quotient ring of $\crr = \tilde{K}^{\oplus k} [D,\sigma ]$ by $S=\{D^k|k\ge 0\}$. And the ring of skew pseudo-differential operators 
	$$
	E_k:=\tilde{K}[\Gamma_1, A_1]((\tilde{D}^{-1}))=\{\sum_{l=M}^{\infty}P_l\tilde{D}^{-l} | \quad P_l\in \tilde{K}[\Gamma_1, A_1]\} \simeq \tilde{K}^{\oplus k}[\Gamma_1]((\tilde{D}^{-1}))
	$$
with the commutation relations 
$$
\tilde{D}^{-1}a=\sigma (a)\tilde{D}^{-1}, \quad a\in \tilde{K}[\Gamma_1, A_1] \quad \mbox{where \quad }
\sigma (A_1)= \xi^{-1} A_1, \quad \sigma (\Gamma_1)=\Gamma_1+1.
$$	
	
	$\widehat{Hcpc}_B(k)$ is the $\tilde{K}$-subalgebra in $\hat{D}_1^{sym}\hat{\otimes}\tilde{K}$ consisting of operators whose homogeneous components are HCPs totally free of $B_j$. 
	
	$$\Phi : \tilde{K}[A_1,\ldots ,A_{k-1}]\rightarrow \tilde{K}^{\oplus k} , \quad   P\mapsto (\sum_i p_i\xi^{i}, \ldots , \sum_i p_i\xi^{i(k-1)})
$$
is an isomorphism of $\tilde{K}$-algebras. 	

The map 
$$
\hat{\Phi}: \widehat{Hcpc}_B(k) \hookrightarrow E_k
$$
defined on monomial HCPs from $\widehat{Hcpc}_B(k)$ as $\hat{\Phi} (a A_j\Gamma_iD^l):=a \Phi (A_j)\Gamma_i\tilde{D}^l$ and extended by linearity on the whole $\tilde{K}$-algebra $\widehat{Hcpc}_B(k)$, is an embedding of $\tilde{K}$-algebras. 
	
	Suppose $B$ is a commutative sub-algebra of $D_1$, then $(C,p,\cf)$ stands for the spectral data of $B$ (the spectral curve, point at infinity and the spectral sheaf with vanishing cohomologies).
	
	The classical ring of pseudo-differential operators  is defined as
	$$
E=K[[x]]((\partial^{-1})).
	$$

There is an isomorphism of $\tilde{K}$-algebras $\psi:\mathfrak{B}\rightarrow M_k(C(\mathfrak{B}))$, where $C(\mathfrak{B})\simeq \tilde{K}[\tilde{D}^k, \tilde{D}^{-k}]$, ($\tilde{K}$ is diagonally embedded into $\tilde{K}^{\oplus k}$):
	$$
	\psi\begin{pmatrix}
		h_0 \\
		h_1 \\
		\cdots \\
		h_{k-1}
	\end{pmatrix}=\begin{pmatrix}
		h_0 &  &  &  \\
		& h_1 &  &  \\
		&  & \cdots &  \\
		&  &  & h_{k-1}
	\end{pmatrix}\quad \psi(D)=T:=\begin{pmatrix}
		& 1 &  & \cdots &  \\
		&  & 1 & \cdots &  \\
		\cdots & \cdots & \cdots & \cdots & \cdots \\
		&  &  & \cdots & 1 \\
		D^k &  &  & \cdots & 
	\end{pmatrix}
	$$
	with $\psi(D^l)=T^l$, and extended by linearity. The map $\psi$ can be obviously extended to 
	$$
\psi : \tilde{K}^{\oplus k}((\tilde{D}^{-1})) \hookrightarrow M_k(\tilde{K}((\tilde{D}^{-k}))).
	$$

Elements of the centralizer $C(\partial^k)$ embedded to $\mathfrak{B}\subset E_k$ via $\hat{\Phi}$ we'll call as a {\it vector form} presentation, and the same elements embedded to $M_k(\tilde{K}[D^k])$ via $\psi\circ \hat{\Phi}$ -- as a {\it matrix form} presentation. 

Translating the description of the centralizer $C(\partial^k)$  into the vector form, we get that $\hat{\Phi}(C(\partial^k))$ consists of Laurent polynomials in 
$\tilde{D}$ with coefficients from $K^{\oplus k}$ and with additional conditions: the coefficient $s_i$ at $\tilde{D}^{-i}$, $i>0$, has a shape $s_{i,j}=0$ for $j=0, \ldots , i-1$.

4. In section \ref{S:normal_forsm_for_non-commuting}, suppose $H$ is an operator whose components are all HCP. Then $E(H)$ denotes  the point set where  $f_{l,i;r}\neq 0$, $v_{\sigma,\rho}$ stands for the weight degree of $H$, and $f_{\sigma,\rho}$ for the highest terms  associated to $(\sigma,\rho)$:
	$$
	v_{\sigma,\rho}(H)=\sup\{\sigma l+\rho j|(l,j)\in E(H)\} \quad
	f_{\sigma,\rho}(H)=\sum_{(l,j)\in E(H,\sigma,\rho)}\sum_{i}f_{l,i;j}\Gamma_lA_{k,i}D^{j} 
	$$
	The {\it up-edge} of the Newton region of $P$ is the set
	$$
	Edg_u(P):=\{(a,b)\in E(P)| \quad a=Sdeg_A(P_b) \mbox{\quad and\quad }  \forall b'>b \quad Sdeg_A(P_{b'})< a  \}.
	$$
	And $H_{d;(\sigma,\rho)}(H), HS_{d;(\sigma,\rho)}^m(H)$ stands for 
	$$
	H_{d;(\sigma,\rho)}(L):=\sum_{\sigma l+\rho j\ge d}\sum_{i=0}^{k-1}\alpha_{l,i;j}\Gamma_lA_i\partial^j
	$$
	and 
	$$
	HS_{d;(\sigma,\rho)}^m(L):=\sum_{\sigma l+\rho j\ge d;l\leq m}\sum_{i=0}^{k-1}\alpha_{l,i;j}\Gamma_l\partial^j
	$$

\section{Normal forms for non-commuting operators}
\label{S:normal_forsm_for_non-commuting}

\subsection{A Newton Region of operators with the property $A_{q}(k)$}
\label{S:General_Newton}

Let $P,Q$ be a pair of monic differential operators from $D_1$. If $[P,Q]\neq 0$, it is useful to study the normal forms of $P$ with respect to $Q$ more carefully. The well known and useful technical tool -- the Newton polygon of a differential operator from the Weyl algebra - can be naturally defined in our situation and applied to such study. In this section we introduce the notion of a Newton region  --  a generalisation of the Newton Polygon, suitable for operators from $\hat{D}_1^{sym}$ satisfying conditions $A_{q}(k)$, and study its basic properties. In this paper they will be used for the proof of a commutativity criterion in section \ref{S:restriction top line}. Further study of the Newton region and of normal forms will be continued in subsequent works.  

In this section let's fix $k\in \dn$. Let $\xi$ be a $k$-th primitive root of $1$, $\tilde{K}=K[\xi ]$.

\begin{Def}
\label{D:NP(H)}
	Suppose $H\in \hat{D}_1^{sym}\hat{\otimes}_K \tilde{K}$ is a HCP from $Hcpc(k)$, $\Ord(H)=r$, written in the G-form:
$$
H=(\sum_{0\leq i<k}\sum_{0\leq l\leq d_i} f_{l,i;r}\Gamma_lA_i+\sum_{0<j\leq N}g_{j;r}B_{j})D^{r}
$$
We define the set $E(H):=\{(l,r)|\quad \exists i, f_{l,i;r}\neq 0 \}$ ($E(H)=\emptyset$ if all coefficients $f_{l,i;r}$ are equal to zero). 

Suppose now $H\in \hat{D}_1^{sym}$  is such that all homogeneous components $H_i$ are HCPs from $Hcpc(k)$ (for example, $H$ satisfies condition $A_k(q)$). We define the {\it Newton region} $NR(H)$ as the convex hull of the union $E(H):=\cup_i E(H_i)$ (i.e. the region can be unbounded). 

We'll say that the point $(a,b)\in E(H_b)\subseteq E(H)$ {\it does not contain $A_i$} if the coefficients $f_{a,i;b}$ of the G-form of $H_b$ satisfy the following property: $f_{a,i;b}=0$ for $i>0$. 

We'll call HCP of the form $f_{l,i;r}\Gamma_lA_iD^r$ or $g_{j;r}B_{j}D^{r}$ as {\it monomials} (of $H$). We'll call HCP of the form $f_{l,i;r}\Gamma_lA_iD^r$ as {\it monomials corresponding to the point } $(l,r)$. 
\end{Def}

\begin{rem}
\label{R:NP(H)}
This definition slightly differs from the well known definition of the Newton polygon of an operator from the Fist Weyl Algebra $A_1$, since the points of the Newton region belong to the $XY$-plane where the $X$-axis stand now for powers of $x\partial$ (hence $X$ equals to $Sdeg_A$), and the $Y$-axis stand for the homogeneous order $\Ord$. Notice that the Newton Polygon of a HCP $H$ will belong to the line $Y=\Ord(H)$. 
\end{rem}

\begin{Def}
\label{D:homogeneous (highest) term}
Suppose  $H\in \hat{D}_1^{sym}\hat{\otimes}_K \tilde{K}$ is such that all homogeneous components $H_i$ are HCPs from $Hcpc(k)$ (for example, $H$ satisfies condition $A_k(q)$). For a real pair $(\sigma,\rho)$ with $\sigma\ge 0$, $\rho >0$ we define: 
	$$
	v_{\sigma,\rho}(H)=\sup\{\sigma l+\rho j|(l,j)\in E(H)\}, \quad 
		E(H,\sigma,\rho)=\{(l,j)\in E(H)|v_{\sigma,\rho}(H)=\sigma l +\rho j \} ,
	$$ 
where we define $v_{\sigma,\rho}(H):=-\infty$ if $E(H)=\emptyset$, and $E(H,\sigma,\rho):=\emptyset$ if $v_{\sigma,\rho}(H)=\infty$ (note that the set $E(H,\sigma,\rho)$ can be empty also if $v_{\sigma,\rho}(H)<\infty$).

	If $E(H,\sigma,\rho)\neq \emptyset$, we define the operator 
		$$
		f_{\sigma,\rho}(H)=\sum_{(l,j)\in E(H,\sigma,\rho)}\sum_{i}f_{l,i;j}\Gamma_lA_{k,i}D^{j} 
		$$
		which is called the {\it homogeneous (highest) term of $H$ associated to} $(\sigma,\rho)$, and the line $l_0:\sigma X+\rho Y=v_{\sigma,\rho}(H)$ is called the {\it $(\sigma,\rho)$-top line}. 
		
		If $E(H,\sigma,\rho)= \emptyset$, we define $f_{\sigma,\rho}(H):=0$. 
\end{Def}

\begin{rem}
	In the following discussion the top line (of a monic operator) will usually go across some  vertex $(0,p)$.
	
	Note  that immediately from definition it follows that 
	$$v_{\sigma,\rho}(H)=\sup_{j\in \sdz}\{\sigma Sdeg_A(H_j)+\rho j \} .$$
	In particular, if $H$ satisfies condition $A_k(0)$, then there exists $(\sigma ,\rho )$ with $\sigma >0$ such that $v_{\sigma,\rho}(H)< \infty$ (e.g. $(1,1)$). 
\end{rem}

The specific basic properties of the Newton region somewhat similar to analogous properties of the Newton polygons from the paper \cite{Dixmier} are collected in the Appendix.

Further we'll need several  statements about the top lines of operators satisfying conditions $A_k(0)$.

\begin{Def}
\label{D:restriction top line}
	Suppose $P\in \hat{D}_1^{sym}\hat{\otimes}_K \tilde{K}$ satisfies condition $A_k(0)$, $\Ord(P)=p$.  A $(\sigma ,\rho )$-top line  which goes across $(0,p)\in E(P)$ and contains at least two vertices is called {\it a restriction top line of $NR(P)$}.
\end{Def}

\begin{rem}
\label{R:restriction top line}
The restriction top line is uniquely defined if it exists. To show this first note that any real pair $(\sigma ,\rho )$ with $\sigma\ge 0$, $\rho >0$ is proportional to some pair $(\tilde{\sigma} ,1)$, and we can consider only such pairs without loss of generality. 

If $(\sigma ,1)$-top line is a restriction top line, then it contains the vertex $(0,p)$ and another vertex, say $(l,j)$, with $j<p$, and $\sigma l+j=p$. If $\sigma'>\sigma$, then  $(\sigma' ,1)$-top line can not be a restriction top line, because $\sigma' l+j>\sigma l+j=p$, i.e. it can not go across $(0,p)$.  Thus, there exists only one pair $(\sigma ,1)$ such that $(\sigma ,1)$-top line is a restriction top line.

As we have noted before, the restriction top line is no longer a trivial notion. Since $Sdeg_A$ might go to infinity, an operator may not have restriction top line at all.
\end{rem}

\begin{Def}
\label{D:asymptotic top line}
	Suppose $P\in \hat{D}_1^{sym}\hat{\otimes}_K \tilde{K}$ satisfies condition $A_k(0)$, $\Ord(P)=p$. If $P$ doesn't have the restriction top line but there exists a top line $l_0:\sigma_0 X+Y=p$, $\sigma_0>0$, such that for any $\sigma>\sigma_0$ the line $l:\sigma X+Y =p$ is not the top line of $N(P)$, we call this top line $l_0$ as {\it the asymptotic top line}.
\end{Def}

For the next lemma we extend the definition of the function $Sdeg_A$ to operators satisfying condition $A_k(0)$ in an obvious way: $Sdeg_A(P):=\sup_{i\in \sdz} Sdeg_A(P_i)$. Of course, for a generic operator $Sdeg_A(P)=\infty$. 

\begin{lemma}
	\label{L: res or asym}
	Suppose $P\in \hat{D}_1^{sym}\hat{\otimes}_K \tilde{K}$ satisfies condition $A_k(0)$, $\Ord(P)=p$.
	 Then only one of the following conditions holds:
	\begin{enumerate}
		\item $Sdeg_A(P)=0$.
		\item $Sdeg_A(P)>0$, and $P$ has the restriction top line.
		\item $Sdeg_A(P)>0$, and $P$ has the asymptotic top line.  
	\end{enumerate}
	In particular, the asymptotic top line is uniquely defined if it exists.
\end{lemma}  

\begin{proof}
Suppose $Sdeg_A(P)=0$. Then for any pair $(\sigma ,\rho )$ with $\sigma \ge 0$, $\rho >0$ we have $v_{\sigma ,\rho }(P)=\rho p$, and then, clearly, any $(\sigma ,\rho )$-top line is not the restriction top line and not an asymptotic top line, because the set $E(P)$ lies on the line $X=0$. 

	Suppose $Sdeg_A(P)>0$. Then, since $P$ satisfies condition $A_k(0)$, the line $l:X+Y=p$ is the $(1,1)$-top line of $P$. Put
	$$
	\sigma_0=\inf\{\sigma|\quad \sigma X+Y=p\quad  \text{is the $(\sigma , 1)$-top line of $P$} \}\ge 1.
	$$
	It is well-defined (finite) since $Sdeg_A(P)>0$. Now consider the line $l_0:\sigma_0 X+Y=p$. If there are more than one vertex on this line, then this line is the restriction top line, and if there is only one point $(0,p)$, then it is the asymptotic top line.
\end{proof}

\begin{ex}
\label{Ex:illustration}

Suppose $P\in \hat{D}_1^{sym}\hat{\otimes}_K \tilde{K}$, $\Ord (P)=p$, satisfies condition $A_k(0)$ and \\ $Sdeg_A(P_{\Ord (P)-i})=i-1$ for all $i>0$ (such condition holds for an operator $P'$ from \cite[Cor.2.4]{GZ24}, which comes from  a {\it generic} pair of operators $P,Q\in D_1$). 

Then it's easy to see that $P$ doesn't have the restriction top line, but the top line $l_0:X+Y=p$ is the asymptotic line.
\end{ex}

\begin{Def}
\label{D:up-edge}
Suppose $P\in \hat{D}_1^{sym}\hat{\otimes}_K \tilde{K}$ satisfies condition $A_k(0)$. We define {\it the up-edge} of the Newton region of $P$ as the set
	$$
	Edg_u(P):=\{(a,b)\in E(P)| \quad a=Sdeg_A(P_b) \mbox{\quad and\quad }  \forall b'>b \quad Sdeg_A(P_{b'})< a  \}.
	$$
\end{Def}

\begin{figure}[!h]
	\label{F:Poly P}
	\centering
	\begin{tikzpicture}
		\draw[line width=1pt, ->](0,0) -- (6,0) node[right] {$Sdeg_A$};
		\draw[line width=1pt, ->](0,0)  -- (0,5) node[above] {$\Ord$};
		\filldraw[black](2,1.5) circle(2pt);
		\filldraw[black](1,3) circle(2pt);
		\filldraw[black](0.5,3) circle(2pt);
		\filldraw[black](0,4.5) circle(2pt);
		\filldraw[black](1.5,2.25) circle(2pt);
		\filldraw[black](2.5,1) circle(2pt);
		\filldraw[black](2.25,1) circle(2pt);
		\draw[dashed](0,4.5) node[left] {$(0,p)$} -- (3.5,0);
		\draw[line width=1pt , red](0,3)  -- (1,3);
		\draw[line width=1pt , red](1,3)  -- (1,2.25);
		\draw[line width=1pt , red](1.5,2.25)  -- (1,2.25);
		\draw[line width=1pt , red](1.5,2.25)  -- (1.5,1.5);
		\draw[line width=1pt , red](1.5,1.5)  -- (2,1.5);
		\draw[line width=1pt , red](2,1.5)  -- (2,1);
		\draw[line width=1pt , red](2,1)  -- (2.5,1);
		\draw[line width=1pt , red](2.5,1)  -- (2.5,0);
	\end{tikzpicture}
	\caption{The up-edge of $P$ with the asymptotic line across $(0,p)$}
\end{figure}
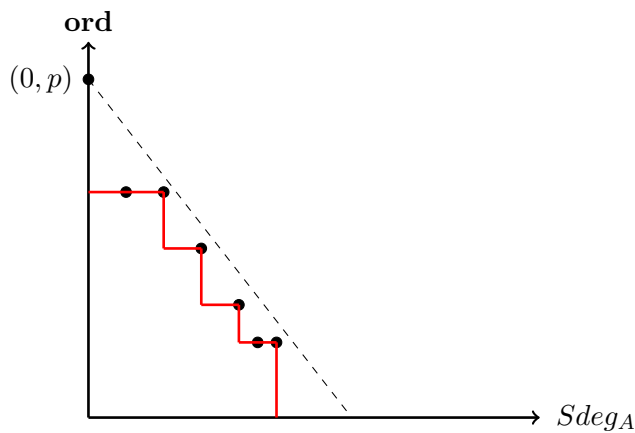

\begin{lemma}
\label{L:up-edge}
Suppose $Q\in D_1$ is a monic operator with  $\Ord (Q)=\deg (Q)=q=k> 0$. 
	Suppose $P\in {D}_1$ has constant highest symbol (cf. \cite[Th. 2.2]{GZ24}), $\Ord (P)=\deg (P)=p$. 
	 Put $P'=S^{-1}PS$, where $S$ is a Schur operator for $Q$ (cf. \cite[Prop. 2.5]{GZ24}). Suppose $(a,b)\in Edg_u(P')$. 
	
	Then the point $(a,b)$  doesn't contain $A_i$.
\end{lemma}

\begin{proof}
By \cite[Cor. 2.4, Th.2.2]{GZ24}  the operator $P'$ satisfies condition $A_q(0)$.

	Suppose $(a,b)\in Edg_u(P')$, and the coefficient at $\Gamma_aD^b$ of the G-form of $P_b'$ is $t=\sum t_iA_{q;i}$, $t_i\in \tilde{K}$. Consider the operator 
	$$
	\tilde{P}:=(ad\partial^q)^a(P')
	$$
	Since $Sdeg_A(P_j')< a$ for all $j>b$, $Sdeg_A(P_b')=a$, and $\partial^q$ commutes with all $A_{q;i}$, we have  $\Ord (\tilde{P})=b+qa$. Besides, $Sdeg_A(\tilde{P}_{b+qa})=0$ and $\tilde{P}_{b+qa}=\lambda t$, $\lambda \in \dq$. 
		
	On the other hand, we know 
	$$
	S^{-1}(ad (Q))^a(P)S=(ad(S^{-1}QS))^a(S^{-1}PS)=\tilde{P},
	$$
	hence we know $\bar{P}:=S\tilde{P}S^{-1}\in D_1$, and $\Ord (\bar{P}=\Ord (\tilde{P})$. Since $S_0=(S^{-1})_0=1$, we get $\lambda t=\bar{P}_{b+qa}=\tilde{P}_{b+qa}\in D_1$. But then by \cite[Lemma 2.1]{GZ24} $t_i=0$ for all $i>0$, i.e. $(a,b)$ does not contain $A_i$. 
\end{proof}

Just noting that the points on the $(\sigma,\rho)$-top line will be in $Edg_u(P')$ when $\sigma,\rho>0$, we have the following Corollary.

\begin{cor}
	\label{C:No Ai on top line}
	In the notation of lemma \ref{L:up-edge} suppose  $\sigma,\rho >0$.  Then the points  on the  $(\sigma,\rho)$-top line don't contain $A_i$. 
	
	In particular, if $P'$ has the restriction top line, then the points on it don't contain $A_i$.
\end{cor}

\subsection{One combinatorial lemma}
\label{S:Key combinatorial lemma}

Suppose $A$ is an associative algebra over $K$, $D,L\in A$ are two non-zero elements. Denote by $L^{(0)}:=L$, $L^{(1)}:=[D,L]=adD(L), \ldots, L^{(n)}=(ad(D))^n(L)$. For any $k\in \dn$ the element $(D+L)^k$ can be written in the form (which we'll call the {\it standard form}), where all $L^{(t)}$ stand on the left hand side of powers of $D$:
$$
(D+L)^k=\sum c_{k;t_1,\ldots ,t_m,l} L^{(t_1)}L^{(t_2)}\cdots L^{(t_m)}D^{l}
$$
where $c_{k;t_1,\ldots ,t_m,l}\in K$ are some constant  coefficients, and $m,l,t_i\in \dz_+$. Our task in this section is to determine  such sum form and the coefficients $c_{k;t_1,\ldots ,t_m,l}$ at each position.

Denote by $L^{(t_1,\ldots,t_m)}:=L^{(t_1)}L^{(t_2)}\cdots L^{(t_m)}$, and put  $L^{(t_1,\ldots,t_m)}=1$ if $m=0$.  We'll call  the index $m$ as the {\it multiple index}, and define the {\it partial degree} of $L^{(t_1,\ldots,t_m)}$ as
$$
Pdeg(L^{(t_1,\ldots,t_m)})=t_1+t_2+\ldots+t_m. 
$$
 It is easy to observe that the coefficient at  $D^k$ in $(D+L)^k$ is $1$ so that it's multiple index is 0, but except for $D^k$, the other terms have multiple index more than 0. Denote by $T_{i,j,k}$ the sum of monomials from the coefficient of $D^{k-i} (i>0)$  in $(D+L)^k$ with partial degree $Pdeg(L^{(t_1,\ldots,t_m)})=j\ge 0$.
 
\begin{lemma}{(Combinatorial)}
\label{L:Key combinatorial}
We have 
\begin{equation}
\label{E:Key_comb}
(D+L)^k=D^k+\sum_{i=1}^k\sum_{j=0}^{i-1}T_{i,j,k}D^{k-i},
\end{equation}
where every monomial in $T_{i,j,k}$  has multiple index $m=i-j$, i.e.
	$$
	T_{i,j,k}=\sum_{t_1+\ldots + t_m=j \atop m=i-j} f_{i,j,k}(t_1,\dots,t_m)L^{(t_1,\dots,t_m)}, 
	$$
	where 
	$$
	f_{i,j,k}(t_1,\dots,t_m)=\binom{k}{i}g(t_1,\dots,t_m),
	$$
	where the function $g$ is defined by recursion:
	\begin{enumerate}
		\item For $m=1$  $g(t_1)\equiv 1$.
		\item For any $m$ with $t_1=\ldots=t_m=0$  $g(t_1,\cdots,t_m)=1$.
		\item For $m>1$, when $t_1=0$:
		$$
		g(0,t_2,\ldots,t_m)=g(t_2,\ldots,t_m)+g(0,t_2-1,\ldots,t_m)+\ldots+g(0,t_2,\ldots,t_m-1)
		$$
		\item For $m>1$, when $t_1\ge 1$:
		$$
		g(t_1,t_2,\ldots,t_m)=g(t_1-1,t_2,\ldots,t_m)+g(t_1,t_2-1,\ldots,t_m)+\ldots+g(t_1,t_2,\ldots,t_m-1),
		$$
	\end{enumerate}
	and we assume that $g(t_1,t_2,\ldots,t_m)=0$ if  $t_i<0$ for at least one $i$.
\end{lemma}

\begin{proof} The proof is by induction on $k$. When $k=1$, $(D+L)^k=D+L$, and it's easy to see that $T_{1,0,1}$ satisfies all conditions in the lemma. Now suppose it is true for $k-1$, consider the generic case.  Note that 
$$(D+L)^k=(D+L)(D+L)^{k-1}=(D+L)^{k-1}D+[D,(D+L)^{k-1}]+L(D+L)^{k-1},$$ 
so that all three summands are written in standard form. By induction we have 
\begin{multline}
\label{E:formula}
(D+L)^{k-1}D+[D,(D+L)^{k-1}]+L(D+L)^{k-1}=\\
D^k+ \sum_{i=1}^{k-1}\sum_{j=0}^{i-1}T_{i,j,k-1}D^{k-i} +
\sum_{i=1}^{k-1}\sum_{j=0}^{i-1}[D,T_{i,j,k-1}]D^{k-1-i}+
LD^{k-1}+\sum_{i=1}^{k-1}\sum_{j=0}^{i-1}LT_{i,j,k-1}D^{k-1-i}. 
\end{multline}
Note that for any $t_1,\ldots ,t_m$ we have $[D,L^{(t_1,\dots,t_m)}]= L^{(t_1+1,\dots,t_m)}+\ldots + L^{(t_1,\dots,t_m+1)}$, where all monomials have multiple index $m$, and  $[D,T_{i,j,k-1}]\in T_{i+1,j+1,k}$. Analogously, $T_{i,j,k-1}\in T_{i,j,k}$ and $LT_{i,j,k-1}\in T_{i+1,j,k}$, where the multiple index of $LT_{i,j,k-1}$ is $i-j+1$.  So, all monomials of $T_{i,j,k}$ (for arbitrary $i,j,k$) have the multiple index $i-j$ as claimed, and therefore 
$$
	T_{i,j,k}=\sum_{t_1+\ldots + t_m=j \atop m=i-j} f_{i,j,k}(t_1,\dots,t_m)L^{(t_1,\dots,t_m)}, 
	$$
	for some $f_{i,j,k}(t_1,\dots,t_m)\in K$. Let's calculate $T_{i,j,k}$ explicitly. We can rewrite formula \eqref{E:formula} as
\begin{multline*}
D^k+\sum_{s=1}^{k-1} \sum_{j=0}^{s-1}T_{s,j,k-1}D^{k-s}+\sum_{s=2}^k\sum_{j=0}^{s-2}[D,T_{s-1,j,k-1}]D^{k-s}+LD^{k-1}+\sum_{s=2}^k\sum_{j=0}^{s-2}LT_{s-1,j,k-1}D^{k-s}\\
=D^k + (T_{1,0,k-1}+L)D^{k-1}+ \sum_{s=2}^{k-1}(\sum_{j=0}^{s-2}(T_{s,j,k-1}+[D,T_{s-1,j,k-1}]+LT_{s-1,j,k-1})+T_{s,s-1,k-1})D^{k-s}+\\
\sum_{j=0}^{k-2}([D,T_{k-1,j,k-1}]+LT_{k-1,j,k-1}),
\end{multline*}	 
whence 	we get
\begin{equation}
\label{E:T_10k}
T_{1,0,k}=T_{1,0,k-1}+L=\binom{k}{1}L,
\end{equation}
\begin{equation}
\label{E:T_sjk}
\mbox{for $1<s<k$} \quad T_{s,j,k}=
\begin{cases}
T_{s,j,k-1}+LT_{s-1,j,k-1} & j=0\\
T_{s,j,k-1}+[D,T_{s-1,j-1,k-1}]+LT_{s-1,j,k-1} & 0<j<s-1,\\
[D,T_{s-1,s-2,k-1}]+T_{s,s-1,k-1} & j=s-1
\end{cases}
\end{equation}
\begin{equation}
\label{E:T_kjk}
T_{k,j,k}=
\begin{cases}
LT_{k-1,0,k-1} & j=0\\
LT_{k-1,j,k-1} + [D,T_{k-1,j-1,k-1}] & 0<j<k-1 .\\
[D,T_{k-1,k-2,k-1}] & j=k-1
\end{cases}
\end{equation}	

Now for $j=0$ and $1<s<k$ we get
$$
T_{s,0,k}=\binom{k-1}{s}L^{(0,\ldots ,0)}+\binom{k-1}{s-1}L^{(0,\ldots ,0)}=\binom{k}{s}L^{(0,\ldots ,0)}
$$
as claimed, and for $s=k$ we also get $T_{k,0,k}=L^{(0,\ldots ,0)}$ as claimed. 

For generic $s,j$ we have 
\begin{multline*}
[D,T_{s-1,j-1,k-1}]= \sum_{t_1+\ldots + t_m=j-1\atop m=s-j}\binom{k-1}{s-1}g(t_1, \ldots ,t_m)(L^{(t_1+1,\ldots ,t_m)}+\ldots + L^{(t_1,\ldots , t_m+1)}) =\\
\sum_{t_1'+\ldots + t_m'=j\atop m=s-j} \binom{k-1}{s-1}(g(t_1'-1,t_2', \ldots ,t_m')+\ldots + g(t_1',\ldots, t_{m-1}',t_m'-1))L^{(t_1',\ldots , t_m')} =\\
\sum_{t_1'+\ldots + t_m'=j\atop t_1'\ge 1, m=s-j}\binom{k-1}{s-1} g(t_1',t_2', \ldots ,t_m') L^{(t_1', \ldots ,t_m')} + 
\sum_{t_2'+\ldots + t_m'=j\atop  m=s-j} \binom{k-1}{s-1}  (g(0, t_2'-1,t_3',\ldots ,t_m')+\ldots +\\ g(0,t_2', \ldots , t_m'-1)) L^{(0,t_2',\ldots ,t_m')}
\end{multline*}
and then for $1<s<k$ and $0<j<s-1$ we get from \eqref{E:T_sjk} 
\begin{multline*}
T_{s,j,k}= \sum_{t_1'+\ldots + t_m'=j \atop m=s-j} \binom{k-1}{s} g(t_1', \ldots ,t_m') L^{(t_1', \ldots ,t_m')}+ 
\sum_{t_1'+\ldots + t_m'=j\atop t_1'\ge 1, m=s-j}\binom{k-1}{s-1} g(t_1',t_2', \ldots ,t_m') L^{(t_1', \ldots ,t_m')} + \\
\sum_{t_2'+\ldots + t_m'=j\atop  m=s-j} \binom{k-1}{s-1}  (g(0, t_2'-1,t_3',\ldots ,t_m')+\ldots + g(0,t_2', \ldots , t_m'-1)) L^{(0,t_2',\ldots ,t_m')} +\\
\sum_{t_2'+\ldots + t_m'=j\atop  m=s-j}\binom{k-1}{s-1} g(t_2',t_3',\ldots ,t_m') L^{(0,t_2',\ldots ,t_m')}= \\
\sum_{t_1'+\ldots + t_m'=j\atop t_1'\ge 1, m=s-j}\binom{k}{s} g(t_1',t_2', \ldots ,t_m') L^{(t_1', \ldots ,t_m')} + \sum_{t_2'+\ldots + t_m'=j\atop  m=s-j} \binom{k}{s} g(0,t_2',\ldots ,t_m') L^{(0,t_2',\ldots ,t_m')}=\\
\sum_{t_1'+\ldots + t_m'=j\atop  m=s-j}\binom{k}{s} g(t_1',t_2', \ldots ,t_m') L^{(t_1', \ldots ,t_m')}
\end{multline*}
as claimed.	For $j=s-1$ we get $m=1$ and therefore $T_{s,s-1,k}=\binom{k-1}{s-1}L^{(s-1)}+\binom{k-1}{s}L^{(s-1)} = \binom{k}{s}L^{(s-1)}$ as claimed.

For $s=k$ and $j=k-1$ we get $m=1$ and therefore $T_{k,k-1,k}= L^{(k-1)}$ as claimed. For generic $j$ we have
\begin{multline*}
T_{k,j,k}=\sum_{t_2'+\ldots + t_m'=j\atop  m=k-j} g(t_2', \ldots ,t_m')L^{(0,t_2', \ldots ,t_m')} +
\sum_{t_1'+\ldots + t_m'=j\atop t_1'\ge 1, m=k-j} g(t_1',t_2', \ldots ,t_m') L^{(t_1', \ldots ,t_m')} + \\
\sum_{t_2'+\ldots + t_m'=j\atop  m=k-j}   (g(0, t_2'-1,t_3',\ldots ,t_m')+\ldots + g(0,t_2', \ldots , t_m'-1)) L^{(0,t_2',\ldots ,t_m')} =
\sum_{t_1'+\ldots + t_m'=j\atop  m=k-j} g(t_1',t_2', \ldots ,t_m') L^{(t_1', \ldots ,t_m')}
\end{multline*}
as claimed and we are done. 	
\end{proof}

\subsection{Commutativity criterion for normal forms having the restriction top line}
\label{S:restriction top line}

In this section we'll prove a commutativity criterion for a pair of differential operators whose normal form has the restriction top line.

Before we formulate the theorem, we fix the notation and give several additional definitions. Let $(P,Q)\in D_1$ be a monic pair of differential operators, $Q$ is normalized, with  $\Ord (Q)=\deg (Q)=q> 0$, $\Ord (P)=\deg (P)=p$. Put $Q'=S^{-1}QS=\partial^q$, $P'=S^{-1}PS$, where $S$ is a Schur operator for $Q$. By \cite[Cor. 2.4]{GZ24}  $P'$ satisfies condition $A_q(0)$, i.e. in particular all its homogeneous components are totally free of $B_j$. 

Assume $F\in K[X,Y]$ is a non-zero polynomial such that $F(P,Q):=\sum_{i,j}c_{i,j}P^iQ^j=0$. Then $F$ can be presented as a sum of $(p,q)$-homogeneous polynomials: $F=F_1+\ldots + F_N$, where
$$
F_{j}(X,Y):=k_1^{(j)}X^{u_1^{(j)}}Y^{v_1^{(j)}}+\dots+k_{m^{(j)}}^{(j)}X^{u_{m^{(j)}}^{(j)}}Y^{v_{m^{(j)}}^{(j)}}, \quad k_{i^{(j)}}^{(j)}\in K,
$$ 
$$
N_{F_j}:=pu_1^{(j)}+qv_1^{(j)}=\cdots=pu_{m^{(j)}}^{(j)}+qv_{m^{(j)}}^{(j)}.
$$ 
Obviously, we have also the equation $F(P',Q')=0$, and since $F(P',Q')\in \hat{D}_1^{sym}$ is an operator whose homogeneous terms are HCPs from $Hcpc(q)$, this equation is equivalent to the system of infinite number of equations on coefficients of homogeneous terms of this operator, written if the G-form. Denote by $f_{l,i;r}(H)$ the coefficient of a HCP $H$ from $Hcpc(q)$. So,
$$
F(P',Q')=0 \quad \Leftrightarrow \quad f_{l,i;r}(F(P',Q')_r)=0, \quad r,l\in \dz, 0\le i< q.
$$
\begin{Def}
\label{D:identity_of_type_i}
We say {\it the identity of type $i$ for $F_j$} holds if 
\begin{equation}
\label{E:identity_of_type_i}
\sum_{1\leq l\leq m^{(j)} }\binom{u^{(j)}_{l}}{i}k_{l}^{(j)}=0.
\end{equation}
\end{Def}

\begin{Def}
	\label{D:H_d}
	Suppose $L$ is a HCP from $Hcpc(q)$ in G-form. For any $\sigma\ge 0,\rho>0$, $d\in \dz$  define "a filtration" of $L$ (determined by the weight function) as 
	$$
	H_{d;(\sigma,\rho)}(L):=\sum_{\sigma l+\rho j\ge d}\sum_{i=0}^{k-1}\alpha_{l,i;j}\Gamma_lA_i\partial^j
	$$
	If there is no ambiguity of $(\sigma,\rho)$, we'll simply write it as $H_d(L)$.
	
	If $L\in \hat{D}_1^{sym}\hat{\otimes}_K \tilde{K}$ and all its homogeneous components $L_i$ are HCP from $Hcpc(q)$ in G-form, we extend definition of $H_{d;(\sigma,\rho)}(L)$ in obvious way. 
\end{Def} 

\begin{lemma}
	\label{L:H_d sum}
	Suppose $L,M\in \hat{D}_1^{sym}\hat{\otimes}_K \tilde{K}$ are two operators such that all homogeneous components $L_i, M_i$ are HCPs from $Hcpc(q)$, suppose $(\sigma,\rho)$ is a real pair with $\sigma\ge 0$, $\rho >0$, and $v_{\sigma,\rho}(L),v_{\sigma,\rho}(M)<\infty$. Then 
	\begin{enumerate}
		\item If $d_1>v_{\sigma,\rho}(L)$, then $H_{d_1}(L)=0$.
		\item $H_d(L+M)=H_d(L)+H_d(M)$.
		\item If $d_1>d_2$, then 
		$$
		v_{\sigma,\rho}(H_{d_2}(L)-H_{d_1}(L))\leq d_1
		$$
		and 
		$$
		H_{d_1}[H_{d_2}(L)]=H_{d_2}[H_{d_1}(L)]=H_{d_1}(L)
		$$
	\end{enumerate}
\end{lemma}
\begin{proof}
	1. $d_1>v_{\sigma,\rho}(L)$, then there doesn't exist $(m,u)\in E(L)$, such that $m\sigma+u\rho\ge d_1$, hence $H_{d_1}(L)=0$.
	
	2, 3 are obvious. 	
\end{proof}

\begin{lemma}
	\label{L:H_d mul}
	Suppose $L,M\in \hat{D}_1^{sym}\hat{\otimes}_K \tilde{K}$ are two operators such that all homogeneous components $L_i, M_i$ are HCPs from $Hcpc(q)$, suppose $(\sigma,\rho)$ is a real pair with $\sigma\ge 0$, $\rho >0$, and $v_{\sigma,\rho}(L),v_{\sigma,\rho}(M)<\infty$. Then 
	\begin{enumerate}
		\item If $d_1\ge v_{\sigma,\rho}(L)$, and $d_2\ge v_{\sigma,\rho}(M)$, then 
		$$
		H_{d_1+d_2}(LM)=H_{d_1+d_2}[H_{d_1}(L)H_{d_2}(M)]
		$$
		\item Suppose $d_1=v_{\sigma,\rho}(L), d_2=v_{\sigma,\rho}(M)$. If $H_{d_1-\sigma}(L)$ and $H_{d_2-\sigma}(M)$ doesn't contain $A_i$, then 
		$$
		H_{d_1+d_2-\sigma}([L,M])=H_{d_1+d_2-\sigma}([H_{d_1-\sigma}(L),H_{d_2-\sigma}(M)])
		$$
		with
		$$
		v_{\sigma,\rho}([L,M])\leq v_{\sigma,\rho}(L)+v_{\sigma,\rho}(M)-\sigma
		$$ 
		\item Suppose $d_1=v_{\sigma,\rho}(L), d_2=v_{\sigma,\rho}(M)$, and $v_{\sigma,\rho}([L,M])\leq d_1+d_2-\sigma$, we have for any $\epsilon>0$,
		$$
		H_{d_1+d_2-\sigma+\epsilon}(LM)=H_{d_1+d_2-\sigma+\epsilon}(ML)
		$$
		In particular
		$$
		H_{d_1+d_2}(LM)=H_{d_1+d_2}(ML)
		$$
	\end{enumerate}
\end{lemma}

\begin{proof}
	1. If $d_1>v_{\sigma,\rho}(L)$ or $d_2>v_{\sigma,\rho}(M)$, then $H_{d_1}(L)=0$ or $H_{d_2}=0$, and by Lemma \ref{L:New Dixmier 2.7} we know there doesn't exist $(l,j)\in E(LM)$, such that $l\sigma+j\rho>d_1+d_2$, hence $H_{d_1+d_2}(LM)=0$. Now let's consider the case when $d_1=v_{\sigma,\rho}(L)$ and $d_2=v_{\sigma,\rho}(M)$. 
	
	Suppose $L_1=H_{d_1}(L)$ and $M_1=H_{d_1}(M)$, put $L_3=L-L_1,M_3=M-M_1$. This means for any $(m_3,u_3)\in E(L_3)$ and $(n_3,v_3)\in E(M_3)$:
	$$
	m_3\sigma+u_3\rho<d_1, \quad n_3\sigma+v_3\rho<d_2
	$$ 
	
	Hence if there exists $(l,j)\in E(L_1M_3)\bigcup E(L_3M_1)\bigcup E(L_3M_3)$, we have $l\sigma+j\rho<d_1+d_2$. This means $H_{d_1+d_2}(L_1M_3)=H_{d_1+d_2}(L_3M_1)=H_{d_1+d_2}(L_3M_3)=0$. Thus
	$$
	H_{d_1+d_2}(LM)=H_{d_1+d_2}(L_1M_1+L_1M_3+L_3M_1+L_3M_3)=H_{d_1+d_2}(L_1M_1)
	$$
	
	2. Assume $L_1=H_{d_1-\sigma}(L), M_1=H_{d_2-\sigma}(M)$, put $L_3=L-L_1,M_3=M-M_1$. Then $v_{\sigma,\rho}(L_3)< d_1-\sigma, v_{\sigma,\rho}(M_3)< d_2-\sigma$ and there doesn't exist $(m_3,u_3)\in E(L_3)$, $(n_3,v_3)\in E(M_3)$, such that 
	$$
	m_3\sigma+u_3\rho=d_1-\sigma, n_3\sigma+v_3\rho=d_2-\sigma
	$$
	By the same arguments as above (use Lemma \ref{L:New Dixmier 2.7} item 1) there doesn't exist \\
	$(l,j)\in E(L_1M_3)\bigcup E(L_3M_1)\bigcup E(L_3M_3)$, such that 
	$$
	l\sigma+j\rho\ge d_1+d_2-\sigma
	$$
	Hence $H_{d_1+d_2-\sigma}(L_1M_3)=H_{d_1+d_2-\sigma}(L_3M_1)=H_{d_1+d_2-\sigma}(L_3M_3)=0$. Thus we get 
	$$
	H_{d_1+d_2-\sigma}(LM)=H_{d_1+d_2-\sigma}(L_1M_1+L_1M_3+L_3M_1+L_3M_3)=H_{d_1+d_2-\sigma}(L_1M_1)
	$$
	For the same reason we have 
	$$
	H_{d_1+d_2-\sigma}(ML)=H_{d_1+d_2-\sigma}(M_1L_1)
	$$
	So we get $H_{d_1+d_2-\sigma}([L,M])=H_{d_1+d_2-\sigma}([L_1,M_1])$.
	
	According to the assumptions, $L_1$ and $M_1$ doesn't contain $A_i$, then by Lemma \ref{L:New Dixmier 2.7} item 3(a), we know 
	$$
	v_{\sigma,\rho}([L_1,M_1])\leq d_1+d_2-\sigma .
	$$
	
	Now suppose $H_1=H_{d_1+d_2-\sigma}([L,M])=H_{d_1+d_2-\sigma}([L_1,M_1])$, and $H_3=H-H_1$. So we have $v_{\sigma,\rho}(H_1)\leq d_1+d_2-\sigma$ and $v_{\sigma,\rho}(H_3)\leq d_1+d_2-\sigma$, hence $v_{\sigma,\rho}([L,M])\leq d_1+d_2-\sigma$.
	
	3. Since $v_{\sigma,\rho}([L,M])\leq d_1+d_2-\sigma$, by Lemma \ref{L:H_d sum}
	$$
	H_{d_1+d_2-\sigma+\epsilon}([L,M])=0
	$$
	Hence 
	$$
	H_{d_1+d_2-\sigma+\epsilon}(LM)-H_{d_1+d_2-\sigma+\epsilon}(ML)=0
	$$
\end{proof}

\begin{rem}
	\label{R: H_d mul}
	 Compare this lemma item 2 with Lemma \ref{L:New Dixmier 2.7} item 3(a). Here we give out a more precise estimation: at that time we need $L,M$ are free of $A_i$, but here we only need a part of them not containing $A_i$.
\end{rem}

Combining this Lemma with  Lemma \ref{L:Key combinatorial}, we get
\begin{cor}
	\label{C:Combinatorial prepare for res infi}
	Suppose $L,M\in \hat{D}_1^{sym}\hat{\otimes}_K \tilde{K}$ are two operators such that all homogeneous components $L_i, M_i$ are HCPs from $Hcpc(q)$, suppose $(\sigma,\rho)$ is a real pair with $\sigma\ge 0$, $\rho >0$, and  $v_{\sigma,\rho}(L)=v_{\sigma,\rho}(M)=p$. Suppose $L,M$ satisfy the condition that 
	$$
	H_{2p}([L,M])=0
	$$
	Then for any $d>0$, we have 
	$$
	H_{dp}((L+M)^d)=\sum_{l=0}^d\binom{d}{l}H_{dp}(M^{d-l}L^l)
	$$
\end{cor}

\begin{proof}
	Apply Lemma \ref{L:Key combinatorial} for $L,M$. Denote  $M^{(0)}=M,M^{(1)}=[L,M], M^{(2)}=[L,[L,M]], \dots$. For $1\leq l\leq d$, since $M^{(l)}=LM^{(l-1)}-M^{(l-1)}L$, then by Lemma \ref{L:H_d mul} item 1, we have   
	$$
	H_{(l+1)p}(M^{(l)})=H_{p}(L)H_{lp}(M^{(l-1)})-H_{lp}(M^{(l-1)})H_p(L)
	$$
	Since $H_{2p}(M^{(1)})=[L,M]=0$, then step by step we will get $H_{(l+1)p}(M^{(l)})=0$. Suppose $(t_1,\dots,t_m)$ and $(i,j)$ are the corresponding index to the term 
	$$
	f_{i,j,d}(t_1,\dots,t_m)M^{(t_1,\dots,t_m)}L^{d-i}
	$$
	in $(L+M)^d$, so by Lemma \ref{L:Key combinatorial}, $i-j=m$ and $j=t_1+\cdots+t_m$. If $j>0$, then at least one of $(t_1,t_2,\dots,t_m)$ are not 0, we have by lemma \ref{L:H_d mul} item 1
	$$
	H_{(j+m)p}(M^{(t_1,\dots,t_m)})= H_{(j+m)p}[ H_{(t_1+1)p}(M^{(t_1)})\times \cdots\times  H_{(t_m+1)p}(M^{(t_m)})]=0
	$$
	Hence for $j>0$
	$$
	H_{dp}(M^{(t_1,\dots,t_m)}L^{d-i})=0
	$$
	So by Lemma \ref{L:H_d sum} item 2 and by Lemma \ref{L:Key combinatorial}, we have 
	\begin{multline*}
		H_{dp}((L+M)^d)=H_{dp}(L^d+\sum_{i=1}^d\sum_{j=0}^{i-1}\sum_{t_1+\ldots + t_m=j \atop m=i-j} f_{i,j,d}(t_1,\dots,t_m)M^{(t_1,\dots,t_m)}L^{d-i})
		\\=H_{dp}(L^d+\sum_{i=1}^d\ f_{i,0,d}(0,\dots,0)M^{(0,\dots,0)}L^{d-i})
	\end{multline*}
	Notice that when $j=0$,$m=i-j=i$. So $M^{(0,\dots,0)}=M^m=M^i$, and $f_{i,0,d}=\binom{d}{i}g(0,\dots,0)$, with $g(0,\dots,0)=1$. Hence
	$$
	H_{dp}((L+M)^d)=\sum_{l=0}^d\binom{d}{l}H_{dp}(M^{d-l}L^l)
	$$
\end{proof}

\begin{rem}
	\label{R:Necessary condition for H2p[L,M]=0}
	The condition $	H_{2p}([L,M])=0$ holds if 
	\begin{enumerate}
		\item $L,M$ doesn't contain $A_i$.
		\item One of $L,M$ is $\partial^{ak}, a\in \mathbb{N}$. 
		\item $\exists r\ge 0$, $H_{p-r}(L)$ and $H_{p-r}(M)$ doesn't contain $A_i$.
	\end{enumerate}
	
	1 can refer to Lemma \ref{L:New Dixmier 2.7} item 3(a). 2 can refer to Lemma \ref{L:New Dixmier 2.7} item 3(b). 3 can be shown  by assuming $L_1=H_{p-r}(L),M_1=H_{p-r}(M)$, and arguing in the same way  like in Lemma \ref{L:H_d mul} item 3, so we omit the details here. Notice that when $r=0$ it's also true.
\end{rem} 

For the proof of our main theorem in this section we need one more definition. 

\begin{Def}
\label{D:H_d^m}
Suppose $L$ is a HCP from $Hcpc(q)$ in G-form. For any $\sigma\ge 0,\rho>0$, $d\in \dz$  define "a filtration" of $H_d(L)$ (determined by the $Sdeg_A$ function) as
	$$
	HS_{d;(\sigma,\rho)}^m(L):=\sum_{\sigma l+\rho j\ge d;l\leq m}\sum_{i=0}^{k-1}\alpha_{l,i;j}\Gamma_l\partial^j
	$$
	If there is no ambiguity of $(\sigma,\rho)$, we'll simply write it as $HS_d^m(L)$.
	
	If $L\in \hat{D}_1^{sym}\hat{\otimes}_K \tilde{K}$ and all its homogeneous components $L_i$ are HCP from $Hcpc(q)$ (in G-form), we extend definition of $HS_{d;(\sigma,\rho)}^m(L)$ in obvious way.
\end{Def} 

By definition,
$$
Sdeg_A(HS^m_d(L))\leq m
$$
and we have 
\begin{lemma}
	\label{L:HS}
	Suppose $L,M\in \hat{D}_1^{sym}\hat{\otimes}_K \tilde{K}$ are two operators such that all homogeneous components $L_i, M_i$ are HCPs from $Hcpc(q)$, suppose $(\sigma,\rho)$ is a real pair with $\sigma\ge 0$, $\rho >0$, and $d_1=v_{\sigma,\rho}(L),d_2=v_{\sigma,\rho}(M)$. Then we have:
	\begin{enumerate}
		\item If $d_1=d_2=d$, then 
		$$
		HS_d^m(L)+HS_d^m(M)=HS_d^m(L+M)
		$$
		\item For any $d$, we have 
		$$
		H_d(HS_d^m(L))=HS_d^m(H_d(L))=HS_d^m(L)
		$$
		\item For any $d$, $Sdeg_A(L)\leq a$ iff $HS_{d}^{a}(L)=H_d(L)$
		\item If $Sdeg_A(H_{d_1}(L))=a_1, Sdeg_A(H_{d_2}(M))=a_2$, then
		$$
		HS_{d_1+d_2}^{a_1+a_2}(LM)=H_{d_1+d_2}(HS_{d_1}^{a_1}(L)HS_{d_2}^{a_2}(M))
		$$
		\item If $E(L)=\{(a_1,b_1)\}$, where $a_1\sigma+b_1\rho=d_1$, and $HS_{d_2}^{a_2}(M)=0$, then 
		$$
		HS_{d_1+d_2}^{a_1+a_2}(LM)=0
		$$ 
	\end{enumerate}
\end{lemma}

\begin{proof}
	1, 2, 3 are by definitions.
	
	4. By Lemma \ref{L:H_d mul} we have $H_{d_1+d_2}(LM)=H_{d_1+d_2}(H_{d_1}(L)H_{d_2}(M))$. Hence we have 
	\begin{multline*}
	HS_{d_1+d_2}^{a_1+a_2}(LM)=HS_{d_1+d_2}^{a_1+a_2}(H_{d_1+d_2}(LM))=HS_{d_1+d_2}^{a_1+a_2}(H_{d_1}(L)H_{d_2}(M))
	\\=HS_{d_1+d_2}^{a_1+a_2}(HS_{d_1}^{a_1}(L)HS_{d_2}^{a_2}(M))=H_{d_1+d_2}(HS_{d_1}^{a_1}(L)HS_{d_2}^{a_2}(M))
	\end{multline*}
	The first equality is by item 2; The second is Lemma \ref{L:H_d mul}; The last two  are by item 3.
	
	5. $HS_{d_2}^{a_2}(M)=0$, means that for any $(n,v)\in E(H_{d_2}(M))$ holds $n> a_2$. Denote $M_0=\Gamma_{v}D^{n}$. Then 	
	$$
	H_{d_1+d_2}(LM_0)=\alpha_{a_1,b_1}\beta_{n,v}H_{d_1+d_2}(\sum_{l=0}^{n}\binom{n}{l}b_1^l\Gamma_{n+a_1-l}D^{b_1+v})=\alpha_{a_1,b_1}\beta_{n,v}\Gamma_{a_1+n}D^{b_1+v}
	$$
	hence $HS_{d_1+d_2}^{a_1+a_2}(LM)=0$.
\end{proof}

\begin{theorem}
	\label{T:Restriction infinite}
	Assume $(P,Q)\in D_1$ be a monic pair of differential operators, $Q$ is normalized, with  $\Ord (Q)=\deg (Q)=q> 0$, $\Ord (P)=\deg (P)=p$. Put $Q'=S^{-1}QS=\partial^q$, $P'=S^{-1}PS$, where $S$ is a Schur operator for $Q$ (cf. \cite[Prop. 2.5]{GZ24}).
	
	Suppose $P'$ has the restriction top line, then there doesn't exist a non-zero polynomial $F\in K[X,Y]$, such that $F(P,Q)=0$.
\end{theorem}

\begin{rem}
It can be shown that if the normal form of $P$ with respect to $Q$ has the restriction top line, then the normal form of $Q$ with respect to $P$ has the restriction top line too. We are going to clarify the details of this fact in  a subsequent paper. 
\end{rem}

\begin{proof}
	Assume the converse: suppose such $F$ exists. The idea of the proof is to show that the identities of type $i$ holds for $F_1$ for  all $i\gg 0$. This would imply $F_1=0$, a contradiction.\footnote{The same idea works in the case of any Burchnall-Chaundy polynomials. For such polynomials it is just an easy exercise to show that theorem is true either if  $P'$ has the restriction top line or if $P'$ has the asymptotic top line.} 
	
	Arrange the vertices on the restriction top line associated to $(\sigma,1)=(p/q,1)$ as $(0,p)$, $(a_0,b_0)$, $(a_1,b_1),\cdots,(a_n,b_n),\cdots$, with $0<a_0<a_1<\cdots<a_n<\cdots$, the coefficient of $(a_i,b_i)$ is $t_i\in\tilde{K}$ according to Corollary \ref{C:No Ai on top line}. Assume $F_1(P,Q)=f_{p,q}(F)=k_1X^{u_1}Y^{v_1}+\dots+k_mX^{u_m}Y^{v_m}$, where $u_1>u_2>\dots>u_m, k_i\neq 0$,  $N_F=v_{p,q}(F)=u_ip+v_iq$ for all $1\leq i\leq m$. Suppose $\bar{F}=F-F_1$, it's easy to see $f_{p,q}(\bar{F})\leq N_F-1$, so that $H_{N_F}(\bar{F})=0$.
	
	Suppose $P'=\partial^p+L$. Since $P'$ has the restriction top line, we know $v_{\sigma,1}(L)=v_{\sigma,1}(P')=p$. Denote $\bd=\partial^p$, and put $\bl=L$, $\bl_0=\Gamma_{a_0}\partial^{b_0}, \bl_1=f_{\sigma,1}(L)-\bl_0, \bl_2=L-\bl_0-\bl_1$. It's easy to find 
	\begin{equation}
	\label{E:qwerty}
	p=v_{\sigma,1}(\bd)=v_{\sigma,1}(\bl_0)\ge v_{\sigma,1}(\bl_1),\quad p\ge v_{\sigma,1}(\bl_2)
	\end{equation}
	with $H_p(\bl_2)=0$, and also
	$$
	a_0=Sdeg_A(\bl_0)<a_1
	$$
	For $d>0$, consider 
	$$
	H_{pd}(P'^d)=H_{pd}((\bd+\bl_0+\bl_1+\bl_2)^d)=H_{pd}((\bd+\bl_0+\bl_1)^d)+H_{pd}(\sum_{l=1}^{d}\binom{d}{l}(\bd+\bl_0+\bl_1)^{d-l}\bl_2^{l}),
	$$
	where the last equality follows from corollary \ref{C:Combinatorial prepare for res infi}. 
	For any $1\leq l\leq d$, by \eqref{E:qwerty}  and  by Lemma \ref{L:H_d mul} item 1 (used  $d$ times), we have 
	$$
	H_{pd}((\bd+\bl_0+\bl_1)^{d-l}\bl_2^l)=H_{pd}[(H_p(\bd+\bl_0+\bl_1))^{d-l}(H_p(\bl_2))^l]=0
	$$
	So we have 
	$$
	H_{pd}(P'^d)=H_{pd}((\bd+\bl_0+\bl_1+\bl_2)^d)
	$$
	
	Since $Q'=\partial^q$, we have $H_q(Q')=Q'=\partial^q$. For the same reason we have 
	$$
	H_{N_F}(P'^{u_j}Q'^{v_j})=H_{N_F}((\bd+\bl_0+\bl_1)^{u_j}\partial^{v_jq})
	$$
	hence 
	$$
	H_{N_F}(F(P',Q'))=H_{N_F}(F_1(P',Q'))=H_{N_F}[\sum_{j=1}^{m}k_j((\bd+\bl_0+\bl_1)^{u_j}\partial^{v_j})]
	$$
	
	Now use  Corollary \ref{C:Combinatorial prepare for res infi} for $L:=\bd,M:=\bl_0+\bl_1,d=u_j$ (notice they satisfy the condition in Remark \ref{R:Necessary condition for H2p[L,M]=0} for item 2). So we have for any $1\leq j\leq m$:
	$$
	H_{u_jp}((\bd+\bl_0+\bl_1)^{u_j})=\sum_{l=0}^{u_j}\binom{u_j}{l}(\bl_0+\bl_1)^l\bd^{u_j-l}
	$$
	Thus
	\begin{equation}
		\label{E:H Nf L0L1 v1 res inf}
		H_{N_F}(F(P',Q'))=\sum_{j=1}^{m}k_j\cdot H_{N_F}(\sum_{l=0}^{u_j}\binom{u_j}{l}(\bl_0+\bl_1)^l\partial^{N_F-lp})
	\end{equation}
	
	To find the coefficient at $\partial^{N_F}$ in the equation $F(P',Q')=0$ (so, this expression should be zero),  we need to calculate $HS_{N_F}^{0}(F(P',Q'))$. Since $HS_p^0(P')=\partial^p$ and $HS_{q}^0(Q')=\partial^q$, by Lemma \ref{L:HS}, we have  
	$$
	HS_{N_F}^{0}(F(P',Q'))=HS_{N_F}^0(H_{N_F}(F(P',Q')))=\sum_{j=1}^mk_j\partial^{N_F}
	$$
	Thus we get the equation of type 0:
	\begin{equation}
		\label{Eq: Restri type 0}
		\sum_{j=1}^{m}k_j=0
	\end{equation}
	Now suppose  the identities of $0,1,\dots,s-1$ type hold, we use induction to prove the identity of type $s$, i.e
	\begin{equation}
		\label{Eq: Restri type i}
		\sum_{j=1}^m\binom{u_j}{s}k_j=0
	\end{equation}
 Note that 
	\begin{multline*}
		H_{N_F}(F(P',Q'))=
		\sum_{l=0}^{s-1}\sum_{j=1}^{m}\binom{u_j}{l}k_j\cdot H_{N_F}((\bl_0+\bl_1)^l\partial^{N_F-lp})+\sum_{l=s}^{u_j} \sum_{j=1}^{m}\binom{u_j}{l}k_j\cdot H_{N_F}((\bl_0+\bl_1)^l\partial^{N_F-lp})
		\\=\sum_{l=s}^{u_j} \sum_{j=1}^{m}\binom{u_j}{l}k_j\cdot H_{N_F}((\bl_0+\bl_1)^l\partial^{N_F-lp})
	\end{multline*}

	To find the coefficient at $\Gamma_{sa_0}\partial^{N_F-s(p-b)}$,  we need to calculate $HS_{N_F}^{sa_0}(F(P',Q'))$. Notice that both $\bl_0$ and $\bl_1$ lie on $Edg_u(P')$, this means they doesn't contain $A_i$, hence they satisfy the condition item 3 in Remark \ref{R:Necessary condition for H2p[L,M]=0}. Use Corollary \ref{C:Combinatorial prepare for res infi} again for  $L:=\bl_0,M:=\bl_1$, we have  
	$$
	H_{lp}((\bl_0+\bl_1)^l)=H_{lp}[\sum_{h=0}^{l}\binom{l}{h}\bl_0^{l-h}\bl_1^{h}]
	$$
	Since we have $Sdeg_A(\bl_0)=a_0<a_1$, hence $HS_{p}^{a_0}(\bl_0)=HS_{p}^{a_0}(\bl)=\bl_0$, hence $HS_{p}^{a_0}(\bl_1)=HS_{p}^{a_0}(\bl)-HS_{p}^{a_0}(\bl_0)=0$. Now we can use Lemma \ref{L:HS} item 5 (since $v_{\sigma ,1}(\bl_1)\le p$), i.e.
	$$
	HS_{lp}^{la_0}(\bl_0^{l-h}\bl_1^h)=\begin{cases}
		0 & h\neq 0
		\\H_{lp}\bl_0^l &h=0
	\end{cases}
	$$
	For the same reason we have 
	$$
	HS_{N_F}^{sa_0}(\bl_0^{l-h}\bl_1^h\partial^{N_F-lp})=\begin{cases}
		0 & h\neq 0 \quad \text{or}\quad  l>s
		\\ H_{N_F}(\bl_0^l\partial^{N_F-pl}) & h=0\quad l=s
	\end{cases}
	$$
	 So
	\begin{multline*}
		HS_{N_F}^{sa_0}(F(P',Q'))=HS_{N_F}^{sa_0}(H_{N_F}(F(P',Q')))
		= \sum_{j=1}^{m}\binom{u_j}{s}k_j\cdot HS_{N_F}^{sa_0}(\sum_{h=0}^{s}\binom{s}{h}\bl_0^{s-h}\bl_1^{h}\partial^{N_F-sp})+
		\\\sum_{l=s+1}^{u_j} \sum_{j=1}^{m}\binom{u_j}{l}k_j\cdot HS_{N_F}^{sa_0}((\bl_0+\bl_1)^l\partial^{N_F-lp})
		=\sum_{j=1}^{m}\binom{u_j}{s}k_j\bl_0^{s}\partial^{N_F-sp}
	\end{multline*}
	Hence we get the identity for type $s$. Now we know they hold for any positive integer $i$, we have $$\binom{u_1}{i}k_1+\cdots+\binom{u_m}{i}k_m=0$$
	so choose $u_2<i\leq u_1$, and consider corresponding equation, we know only $u_1>i$, hence only one term left, and we get
	$$
	k_1\binom{u_1}{i}=0
	$$
	We get $k_1=0$, this is a contradiction.
\end{proof}

\section{Appendix}
\label{S:appendix}

In this section we collect all necessary basic technical assertions about the function $v_{\sigma,\rho}$ and the homogeneous highest terms $ f_{\sigma,\rho}$ used in the paper.

\begin{lemma}
	\label{L:New Dixmier 2.4}
		Suppose $L,M\in \hat{D}_1^{sym}\hat{\otimes}_K \tilde{K}$ are two operators such that all homogeneous components $L_i, M_i$ are HCPs from $Hcpc(k)$, suppose $(\sigma,\rho)$ is a real  pair with $\sigma\ge 0$, $\rho >0$, and $v_{\sigma,\rho}(L),v_{\sigma,\rho}(M)<\infty$. Then 
		\begin{enumerate}
			\item $v_{\sigma,\rho}(L+M)\leq \max\{v_{\sigma,\rho}(L),v_{\sigma,\rho}(M)\}$, and the equality holds if $v_{\sigma,\rho}(L)\neq v_{\sigma,\rho}(M)$.
			\item If $v_{\sigma,\rho}(L)\neq v_{\sigma,\rho}(M)$,
			then 
			$$ f_{\sigma,\rho}(L+M)=
			\begin{cases}
				f_{\sigma,\rho}(L)  & v_{\sigma,\rho}(L)> v_{\sigma,\rho}(M)
				\\f_{\sigma,\rho}(M)  & v_{\sigma,\rho}(L)< v_{\sigma,\rho}(M)
			\end{cases}
			$$
			so that we have $f_{\sigma,\rho}(L+M)=f_{\sigma,\rho}(f_{\sigma,\rho}(L)+f_{\sigma,\rho}(M))$ if $f_{\sigma,\rho}(L),f_{\sigma,\rho}(M)\neq 0$.
			\item If $v_{\sigma,\rho}(L)=v_{\sigma,\rho}(M)=v_{\sigma,\rho}(L+M)$, then 
			$$
			f_{\sigma,\rho}(L+M)=f_{\sigma,\rho}(L)+f_{\sigma,\rho}(M)
			$$
		\end{enumerate}
\end{lemma}

\begin{proof}
	1. Obviously, for any operator $P$ from formulation we have 
	$$
v_{\sigma,\rho}(P)\ge \sigma \max\{l | (l,j)\in E(P_j)\} + \rho j
	$$
	for any $j\in \dz$. 
	Next, note that for any fixed $j\in \dz$
$$
	\max\{l |(l,j)\in E(L_j+M_j)\}\leq \max \{\max\{l| (l,j)\in E(L_j)\},\max\{l| (l,j)\in E(M_j)\}\} .
$$	

	Let, say,  $v_{\sigma,\rho}(L)=\max\{v_{\sigma,\rho}(L),v_{\sigma,\rho}(M)\}$. Then for any $j\in \dz$
	\begin{multline*}
v_{\sigma,\rho}(L)\ge \sigma \max \{\max\{l| (l,j)\in E(L_j)\},\max\{l| (l,j)\in E(M_j)\}\} + \rho j \ge \\ 
\sigma \max\{l |(l,j)\in E(L_j+M_j)\} + \rho j,
	\end{multline*}
	hence $v_{\sigma,\rho}(L)\ge v_{\sigma,\rho}(L+M)$. 
	
	If, say, $v_{\sigma,\rho}(L)> v_{\sigma,\rho}(M)$, then $\forall \varepsilon >0$ there exist $j\in \dz$ such that 
	$$\sigma \max\{l| (l,j)\in E(L_j)\} + \rho j> v_{\sigma,\rho}(L) -\varepsilon ,$$ 
	and if $\varepsilon$ is sufficiently small, then  $\sigma \max\{l| (l,j)\in E(L_j)\} + \rho j> v_{\sigma,\rho}(M)$. Then for such $\varepsilon$ and $j$ we have $v_{\sigma,\rho}(L_j+M_j)=v_{\sigma,\rho}(L_j)$ and 
	$v_{\sigma,\rho}(L+M)\ge v_{\sigma,\rho}(L_j+M_j)> v_{\sigma,\rho}(L) -\varepsilon$, whence $v_{\sigma,\rho}(L)= v_{\sigma,\rho}(L+M)$.
	
	2. It's obvious.
	
	3. Just by the definition of $f_{\sigma,\rho}$.
\end{proof}

\begin{cor}
	\label{C:H1 H3 decomposition}
	Suppose $0\neq H\in \hat{D}^{sym}\hat{\otimes}_K \tilde{K}$, with all homogeneous components in $H$ are HCPs from $Hcpc(k)$, suppose $(\sigma,\rho)$ is a real  pair with $\sigma\ge 0$, $\rho >0$ and $v_{\sigma , \rho}(H)<\infty$. Suppose $H_1=f_{\sigma,\rho}(H)$, $H_2=H-H_1$. Then one of the following is true:
	\begin{enumerate}
		\item $v_{\sigma,\rho}(H_2)<v_{\sigma,\rho}(H)$;
		\item $v_{\sigma,\rho}(H_2)=v_{\sigma,\rho}(H)$ but $f_{\sigma,\rho}(H_2)=0$.
	\end{enumerate} 
\end{cor}

\begin{proof}
	If $H_1\neq 0$, then  $v_{\sigma,\rho}(H)=v_{\sigma,\rho}(H_1)$, so by Lemma \ref{L:New Dixmier 2.4} item 1, we know $v_{\sigma,\rho}(H_2)\leq v_{\sigma,\rho}(H)$. If $H_1=0$, the equality holds.
	
    If $v_{\sigma,\rho}(H_2)=v_{\sigma,\rho}(H)$, then by the definition of $H_1$ we know there doesn't exist $(l,j)\in E(H_2)$, such that $\sigma l+\rho j=v_{\sigma,\rho}(H)=v_{\sigma,\rho}(H_2)$, so by the definition of $f_{\sigma,\rho}$ we get $f_{\sigma,\rho}(H_2)=0$. 
\end{proof}

We now want to estimate $v_{\sigma,\rho}(LM)$ and $v_{\sigma,\rho}([L,M])$ with the help of $v_{\sigma,\rho}(L)$ and $v_{\sigma,\rho}(M)$ (cf. similar estimations for $L,M\in A_1$ in \cite[L.2.7]{Dixmier}).   We consider first the case when $L,M$ are monomials from $Hcpc(k)$.

\begin{lemma}
	\label{L:New Dixmier 2.7 two monomial case}
	Suppose $L,M\in \hat{D}_1^{sym}\hat{\otimes}_K \tilde{K}$ are two monomial operators from $Hcpc(k)$, suppose $(\sigma,\rho)$ is a real pair with $\sigma\ge 0$, $\rho >0$. Then 
	\begin{enumerate}
		\item $v_{\sigma,\rho}(LM)= v_{\sigma,\rho}(L)+v_{\sigma,\rho}(M)$.
		\item $v_{\sigma,\rho}([L,M])\leq v_{\sigma,\rho}(L)+v_{\sigma,\rho}(M)$.
		In the following cases we have more precise estimation:
		\begin{enumerate}
			\item In the case of $L$ and $M$ don't contain $A_i$, then 
			$$
			v_{\sigma,\rho}([L,M])\leq v_{\sigma,\rho}(L)+v_{\sigma,\rho}(M)-\sigma ;
			$$
			\item Suppose one of $L,M$ is $g\partial^b$, where $b=ck$, $c\in \mathbb{N}$, $g\in K$. Then 
			$$
			v_{\sigma,\rho}([L,M])\leq v_{\sigma,\rho}(L)+v_{\sigma,\rho}(M)-\sigma .
			$$
	\end{enumerate}
\end{enumerate}
\end{lemma}

\begin{proof} If $Sdeg_A(L)= -\infty$ or $Sdeg_A(M)= -\infty$, then $L$ or $M$ depends only on $B_j$, so $LM$ and $ML$ depends only on $B_j$ by formulae (2.6-2.9) from \cite[Lem. 2.10]{GZ24}, and therefore $v_{\sigma ,\rho}(LM)=-\infty$, $v_{\sigma ,\rho}([L,M])=-\infty$, and all statements of lemma are trivial. So, we can assume below that $Sdeg_A(L,M)\neq -\infty$.

	1. Suppose 
	$$
	L=a_{i_1,m}\Gamma_mA_{i_1}D^u, M=a_{i_2,n}\Gamma_nA_{i_2}D^v
	$$ 
	Then 
	\begin{equation}
		\label{E:LM}
		LM=a_{i_1,m}a_{i_2,n}\xi^{ui_2}\sum_{t=0}^{n}\binom{n}{t}u^{n-t}\Gamma_{t+m}A_{i_1+i_2}D^{u+v}+\ldots ,
	\end{equation}
where $\ldots$ here and below in the proof mean terms containing $B_j$ (although this equation may contain terms with $B_j$,  here we are discussing $v_{\sigma,\rho}$, so we don't have to write them out, and for convenience we will always forget about that in the following).
	
	Hence we know $v_{\sigma,\rho}(LM)=\sup\{(l,j)\in E(LM)\}=(m+n)\sigma+(u+v)\rho=v_{\sigma,\rho}(L)+v_{\sigma,\rho}(M)$.
	
	2. Since $v_{\sigma,\rho}(LM)=v_{\sigma,\rho}(ML)=v_{\sigma,\rho}(L)+v_{\sigma,\rho}(M)$, by lemma \ref{L:New Dixmier 2.4} we know $v_{\sigma,\rho}([L,M])\leq v_{\sigma,\rho}(L)+v_{\sigma,\rho}(M)$.
	
	Now consider the precise estimation:
	If $L,M$ both don't contain $A_i$, assume 
	$$
	L=a_1\Gamma_mD^u, M=a_2\Gamma_nD^v
	$$
	then 
	$$
	\begin{cases}
		LM=a_1a_2\sum_{t=0}^{n}\binom{n}{t}u^{n-t}\Gamma_{t+m}D^{u+v}+\ldots 
		\\ML=a_1a_2\sum_{t=0}^{m}\binom{m}{t}v^{m-t}\Gamma_{t+n}D^{u+v} +\ldots 
	\end{cases}
	$$
	Hence
	$$
	[L,M]=a_1a_2(u-v)\Gamma_{m+n-1}D^{u+v}+\cdots ,
	$$
	where $\ldots$ mean terms with the value of $v_{\sigma ,\rho}$ less than $(m+n-1)\sigma +(u+v)\rho =v_{\sigma,\rho}(L)+v_{\sigma,\rho}(M)-\sigma$. 
	Thus $v_{\sigma,\rho}([L,M])\leq v_{\sigma,\rho}(L)+v_{\sigma,\rho}(M)-\sigma$.
	
	If one of $L,M$ is $g\partial^{b}$, say, $L=a_1\Gamma_mA_iD^u$, $M=g\partial^{ck}$, then 
	$$
	\begin{cases}
		LM=a_1g\Gamma_mA_iD^{u+ck} + \ldots 
		\\ML=a_1g\sum_{t=0}^{m}\binom{m}{t}(ck)^{m-t}\Gamma_{t}A_iD^{u+ck}
	\end{cases}
	$$
	Hence
	$$
	[L,M]=-a_1gmckA_i\Gamma_{m-1}D^{u+ck}
	$$
	Thus $v_{\sigma,\rho}([L,M])\leq v_{\sigma,\rho}(L)+v_{\sigma,\rho}(M)-\sigma$.
\end{proof}

Now we come to the general case:
\begin{lemma}
	\label{L:New Dixmier 2.7}
	 Suppose $L,M\in \hat{D}_1^{sym}\hat{\otimes}_K \tilde{K}$ are two operators such that all homogeneous components $L_i, M_i$ are HCPs from $Hcpc(k)$, suppose $(\sigma,\rho)$ is a real pair with $\sigma\ge 0$, $\rho >0$, and $v_{\sigma,\rho}(L),v_{\sigma,\rho}(M)<\infty$. Then 
	\begin{enumerate}
		\item For any $(l,j)\in E(LM)$, there exists $(m,u)\in E(L)$ and $(n,v)\in E(M)$, such that 
		$$
		l\leq m+n, \quad j\leq u+v
		$$ 
		\item $v_{\sigma,\rho}(LM)\leq v_{\sigma,\rho}(L)+v_{\sigma,\rho}(M)$. The equality holds if one of the following case is true:
		\begin{enumerate}
			\item $f_{\sigma,\rho}(L)\neq 0,f_{\sigma,\rho}(M)\neq 0$, with $f_{\sigma,\rho}(L)$ and $f_{\sigma,\rho}(M)$ don't contain $A_i$.
			\item $f_{\sigma,\rho}(L)\neq 0,f_{\sigma,\rho}(M)=0$, with $f_{\sigma,\rho}(L)$ doesn't contain $A_i$ and $\exists \epsilon>0$ such that all points $(l,j)\in E(M)$ with $\sigma l+\rho j>v_{\sigma,\rho}(M)-\epsilon$ don't contain $A_i$.
			\item $f_{\sigma,\rho}(L)=0,f_{\sigma,\rho}(M)\neq 0$ with $f_{\sigma,\rho}(M)$ doesn't contain $A_i$ and $\exists \epsilon>0$ such that all points $(l,j)\in E(L)$ with $\sigma l+\rho j>v_{\sigma,\rho}(L)-\epsilon$ don't contain $A_i$.
			\item $f_{\sigma,\rho}(L)=0,f_{\sigma,\rho}(M)=0$, and $\exists \epsilon>0$ such that all points $(l,j)\in E(L)$ with $\sigma l+\rho j>v_{\sigma,\rho}(L)-\epsilon$ don't contain $A_i$ and all points $(l,j)\in E(M)$ with $\sigma l+\rho j>v_{\sigma,\rho}(M)-\epsilon$ don't contain $A_i$.
		\end{enumerate}
		\item $v_{\sigma,\rho}([L,M])\leq v_{\sigma,\rho}(L)+v_{\sigma,\rho}(M)$
		
		In the following cases we have more precise estimation:
		\begin{enumerate}
			\item In the case of $L$ and $M$ don't contain $A_i$, then 
			$$
			v_{\sigma,\rho}([L,M])\leq v_{\sigma,\rho}(L)+v_{\sigma,\rho}(M)-\sigma
			$$
		  \item Suppose $M=g\partial^n$, where $n=mk$, $m\in \mathbb{N}$, $g\in K$. Then 
		  $$
		  v_{\sigma,\rho}([L,M])\leq v_{\sigma,\rho}(L)+v_{\sigma,\rho}(M)-\sigma
		  $$
	    \end{enumerate}
	\end{enumerate}
\end{lemma}

\begin{proof} If $E(LM)=\emptyset$, there is nothing to prove. So, we can assume $E(LM)\neq \emptyset$. In this case $E(L)\neq \emptyset$ and $E(M)\neq \emptyset$, since otherwise $L$ or $M$ would contain only monomials with $B_j$, and then $LM$ would contain also only monomials with $B_j$ according to formulae (2.6-2.9) from \cite[Lem. 2.10]{GZ24}, i.e. $E(LM)=\emptyset$, a contradiction. 

	1. 	Suppose the result is not true, hence there exists $(l_0,j_0)\in E(LM)$, but for any $(m,u)\in E(L)$ and $(n,v)\in E(M)$, whether $l_0>m+n$ or $j_0>u+v$ holds. Assume $L_0,M_0$ are monomial elements in $L,M$, $L_0=a_{m,i_1;u}\Gamma_mA_{i_1}D^u$, $M_0=a_{n,i_2;v}\Gamma_nA_{i_2}D^v$ (obviously, it's sufficient to consider only monomials corresponding to points $(m,n)$, $(u,v)$). Then like in equation \ref{E:LM} (Lemma \ref{L:New Dixmier 2.7 two monomial case} item 1) we have 
	\begin{equation}
	\label{E:L_0M_0}
	L_0M_0=a_{i_1,m}a_{i_2,n}\xi^{ui_2}\sum_{t=0}^{n}\binom{n}{t}u^{n-t}\Gamma_{t+m}A_{i_1+i_2}D^{u+v} +\ldots 
	\end{equation}
	Hence for any $(l,j)\in E(L_0M_0)$, $l\leq m+n$ and $j\leq u+v$. This means $(l_0,j_0)\notin E(L_0M_0)$ for any monomials of $L,M$, so $(l_0,j_0)\notin E(LM)$, this is a contradiction.
	
	2. We know $v_{\sigma,\rho}(LM)=\sup \{\sigma l+\rho j|(l,j)\in E(LM) \}$. Thus for any $\epsilon>0$, there exists $(l,j)\in E(LM)$, such that $v_{\sigma,\rho}(LM)< \sigma l+\rho j+\epsilon$. According to item 1, there exist $(m,u)\in E(L)$ and $(n,v)\in E(M)$, such that 
	$l\leq m+n$ and $j\leq u+v$, thus we have
	$$
	v_{\sigma,\rho}(LM)< \sigma l+\rho j+\epsilon\leq (\sigma m+\rho u)+(\sigma n+\rho v)+\epsilon\leq v_{\sigma,\rho}(L)+v_{\sigma,\rho}(M)+\epsilon
	$$
	So, we get $v_{\sigma,\rho}(LM)\leq v_{\sigma,\rho}(L)+v_{\sigma,\rho}(M)$.
	
	Now lets discuss when  the equality holds:
	
		(a) $f_{\sigma,\rho}(L)\neq 0,f_{\sigma,\rho}(M)\neq 0$
		
		Suppose $u_0=\sup\{u|(m,u)\in E(f_{\sigma,\rho}(L))  \}$, and $v_0=\sup\{v|(n,v)\in E(f_{\sigma,\rho}(M))  \}$. Notice that $u_0$ is an integer and $u_0\le \frac{v_{\sigma,\rho}(L)}{\rho}$ (because $\rho>0$), so $u_0$ is well-defined, so does $v_0$. And suppose $m_0,n_0$ are the corresponding integers for $u_0$ and $v_0$, such that
		$$
		m_0\sigma +u_0\rho=v_{\sigma,\rho}(L), \quad n_0\sigma +v_0\rho=v_{\sigma,\rho}(M)
		$$ 
		Hence $(m_0,u_0)\in E(L; \sigma ,\rho)$ and $(n_0,v_0)\in E(M; \sigma ,\rho)$. Suppose $L_0=a_{0,m_0}\Gamma_{m_0}D^{u_0}, M_0=a_{0,n_0}\Gamma_{n_0}D^{v_0}$ are the monomials corresponding to the points $(m_0,u_0)\in E(L; \sigma ,\rho)$ and $(n_0,v_0)\in E(M; \sigma ,\rho)$  (they don't contain $A_i$ according to the assumptions).
		
		Now put $L_1=f_{\sigma,\rho}(L), L_2=L_1-L_0,L_3=L-L_1$, then for any $(m,u)\in E(L_2)$, we have $u<u_0$, and for any $(m,u)\in E(L_3)$, we have $m\sigma+u\rho<m_0\sigma+u_0\rho$. For the same we assume $M_1=f_{\sigma,\rho}(M), M_2=M_1-M_0, M_3=M-M_1$, for any $(n,v)\in E(M_2)$, we have $v<v_0$, and for any $(n,v)\in E(M_3)$, we have $n\sigma+v\rho<n_0\sigma+v_0\rho$. Thus we get the decomposition 
		$$
		L=L_0+L_2+L_3, \quad M=M_0+M_2+M_3
		$$
		
		Consider the following equation:
		$$
		LM=L_0M_0+L_0(M_2+M_3)+(L_2+L_3)M_0+(L_2+L_3)(M_2+M_3)
		$$
		We want to show $(m_0+n_0,u_0+v_0)\in E(LM)$. This can be true if $(m_0+n_0,u_0+v_0)\in E(L_0M_0)$, but doesn't appear in the rest three terms:
		
		By formula \eqref{E:L_0M_0} we know $(m_0+n_0,u_0+v_0)\in E(L_0M_0)$. 
		
		On the other hand, in $L_0M_2$, since for any $(n,v)\in E(M_2)$ we have $v<v_0$, thus for any $(l,j)\in E(L_0M_2)$ we have $j<v_0+u_0$, hence $(m_0+n_0,u_0+v_0)\notin E(L_0M_2)$. Thus there doesn't exist $(n,v)\in E(M_2)$ such that  
		$$
		n_0\leq n, \quad v_0\leq v .
		$$
		
		Also for $L_0M_3$, since for any $(n,v)\in E(M_3)$, we have $n\sigma+v\rho<n_0\sigma+v_0\rho$, we also have  there doesn't exist $(n,v)\in E(M_3)$ such that  
		$$
		n_0\leq n, \quad v_0\leq v .
		$$
		Then according to  item 1, we know $(m_0+n_0,u_0+v_0)\notin E(L_0(M_2+M_3))$, since, obviously, $E(M_2+M_3)\subseteq E(M_2)\cup E(M_3)$. The same arguments work for $(L_2+L_3)(M_0)$ and $(L_2+L_3)(M_2+M_3)$. So, we get 
		$$
		(m_0+n_0,u_0+v_0)\notin E(L_0(M_2+M_3))\cup E((L_2+L_3)(M_2+M_3))\cup E((L_2+L_3)M_0).
		$$
		
		Hence we have $(m_0+n_0,u_0+v_0)\in E(LM)$, this means
		$$
		v_{\sigma,\rho}(LM)\ge (m_0+n_0)\sigma+(u_0+v_0)\rho=v_{\sigma,\rho}(L)+v_{\sigma,\rho}(M)
		$$
		and together with $v_{\sigma,\rho}(LM)\leq v_{\sigma,\rho}(L)+v_{\sigma,\rho}(M)$ we get  the equality.
		
		(b) $f_{\sigma,\rho}(L)\neq 0,f_{\sigma,\rho}(M)=0$
		
		It's easy to see the equality holds iff the following is true
		$$
		v_{\frac{\sigma}{\rho},1}(LM)=v_{\frac{\sigma}{\rho},1}(L)+v_{\frac{\sigma}{\rho},1}(M)
		$$
		So here we may assume $\rho=1$.
		
		Since $f_{\sigma,\rho}(M)=0$, then for any $\epsilon>0$, there exists $(n,v)\in E(M)$, such that 
		$$
		n\sigma+v<v_{\sigma,1}(M)<n\sigma+v+\epsilon
		$$
		So we can choose 
		$$
		v_0=\sup\{v|(n,v)\in E(M), n\sigma+v>v_{\sigma,1}(M)-\epsilon \},
		$$ 
		where $\epsilon< \epsilon_0$ and $\epsilon_0$ is the number that all points $(n,v)\in E(M)$ with $\sigma n+v>v_{\sigma,\rho}(M)-\epsilon_0$ doesn't contain $A_i$ as in assumption. This $v_0$ is well defined since $\{v|(n,v)\in E(M), n\sigma+v>v_{\sigma,1}(M)-\epsilon \}$ is a non-empty set and $v<v_{\sigma,1}(M)$ always holds.
		And we choose $n_0:=\sup\{n|(n,v_0)\in E(M) \}$, it's easy to see $n_0$ is well-defined and $(n_0,v_0)$ satisfies the properties:
		\begin{enumerate}
			\item [(1)] $(n_0,v_0)\in E(M)$, with $v_{\sigma,\rho}(M)-\epsilon <n_0\sigma+v_0<v_{\sigma,1}(M) $
			\item [(2)] Suppose the monomial corresponding to  $(n_0,v_0)$  is  
			$$
			M_0=a_{n_0,v_0}\Gamma_{n_0}D^{v_0}
			$$
			$$
			M_1=\sum_{(n,v)\in E(M)| n\sigma+v>v_{\sigma,1}(M)-\epsilon }a_{n,v}\Gamma_nD^v
			$$
			($M_1$ is well-defined and it doesn't contain $A_i$). Define $M_2=M_1-M_0$ Then for any $(n,v)\in E(M_2)$, we have either $n\sigma+v\le n_0\sigma+v_0$ or $v<v_0$.
			\item [(3)] Suppose $M_3=M-M_1$, then $v_{\sigma,1}(M_2)<n_0\sigma+v_0$.
		\end{enumerate}
		
		Since $f_{\sigma,1}(L)\neq 0$, we can define $L_0,L_1,L_2,L_3$ in the same way like in (a). Then again 
		$$
		LM=L_0M_0+L_0(M_2+M_3)+(L_2+L_3)M_0+(L_2+L_3)(M_2+M_3).
		$$
		
		For the same reason we know $(m_0+n_0,u_0+v_0)\in E(LM)$, because $(m_0+n_0,u_0+v_0)\in E(L_0M_0)$, but doesn't appear in the rest three parts. Thus $(m_0+n_0,u_0+v_0)\in E(LM)$, and 
		$$
		v_{\sigma,\rho}(LM)\ge (m_0+n_0)\sigma+(u_0+v_0)\ge v_{\sigma,1}(L)+v_{\sigma,1}(M)-\epsilon .
		$$
		Together with the inequality from item 2) we get the equality. 
	
	    (c) $f_{\sigma,\rho}(L)=0,f_{\sigma,\rho}(M)\neq 0$. This case is analogous to b), so  
		we omit the details.
		
		(d) $f_{\sigma,\rho}(L)=0,f_{\sigma,\rho}(M)=0$, in this case just deal with $L,M$ like in (b), the discussion will be the same, we omit the details.

	3. The inequality is obvious in view of item 2. 
		
	3(a). Assume the converse, i.e. $v_{\sigma,\rho}([L,M])>v_{\sigma,\rho}(L)+v_{\sigma,\rho}(M)-\sigma$, then there exist $(l,j)\in E([L,M])$, such that $l\sigma+j\rho>v_{\sigma,\rho}(L)+v_{\sigma,\rho}(M)-\sigma$. 
	
	Suppose $L_0=a_{m,u}\Gamma_mD^u, M_0=a_{n,v}\Gamma_nD^v$(according to the assumptions they don't contain $A_i$) are the monomials in $L,M$. Using the calculation in Lemma \ref{L:New Dixmier 2.7 two monomial case} item 2, we have 
	$$
	[L_0,M_0]=a_{m,u}a_{n,v}(u-v)\Gamma_{m+n-1}D^{u+v}+\cdots
	$$
	This means for any $(l_0,j_0)\in E([L_0,M_0])$, 
	$$l_0\sigma+j_0\rho\leq (m+n-1)\sigma+(u+v)\rho \le v_{\sigma ,\rho} (L)+v_{\sigma,\rho} (M)-\sigma ,$$ 
	but $l\sigma+j\rho>v_{\sigma,\rho}(L)+v_{\sigma,\rho}(M)-\sigma$,this means $(l,j)\notin E([L_0,M_0])$ for any $L_0,M_0$, Hence $(l,j)\notin E([L,M])$, a contradiction.
	
	3(b) The arguments are the same as in 3(a), we omit the proof here.
\end{proof}

\begin{lemma} 
	\label{L:f sigma,rho LM}
	In the notations of lemma \ref{L:New Dixmier 2.7}, if $v_{\sigma,\rho}(LM)= v_{\sigma,\rho}(L)+v_{\sigma,\rho}(M)$, and $f_{\sigma,\rho}(LM)\neq 0$, then we have 
\begin{equation}
	\label{E:v sigma rho LM}
	v_{\sigma,\rho}[f_{\sigma,\rho}(LM)]= v_{\sigma,\rho}(L)+v_{\sigma,\rho}(M)
\end{equation}
On the other hand, if (\ref{E:v sigma rho LM}) holds and $v_{\sigma ,\rho}(L)\neq -\infty$ and $v_{\sigma ,\rho}(M)\neq -\infty$, then $f_{\sigma,\rho}(L)\neq 0, f_{\sigma,\rho}(M)\neq 0$ and $v_{\sigma,\rho}(LM)= v_{\sigma,\rho}(L)+v_{\sigma,\rho}(M)$.	
\end{lemma}

\begin{proof}
	 If $v_{\sigma,\rho}(LM)= v_{\sigma,\rho}(L)+v_{\sigma,\rho}(M)$ and $f_{\sigma,\rho}(LM)\neq 0$, then  there exist $(l,j)\in E(f_{\sigma,\rho}(LM))\subseteq E(LM)$ such that 
	$$
	\sigma l+\rho j=v_{\sigma,\rho}(f_{\sigma ,\rho}(LM))=v_{\sigma,\rho}(LM)=v_{\sigma,\rho}(L)+v_{\sigma,\rho}(M). 
	$$

	Assume now \ref{E:v sigma rho LM} holds. Then $f_{\sigma,\rho}(LM)\neq 0$ (as $v_{\sigma ,\rho}(L)\neq -\infty$ and $v_{\sigma ,\rho}(M)\neq -\infty$).
	Define $H=LM, H_1=f_{\sigma,\rho}(LM), H_2=H-H_1$ like in Corollary \ref{C:H1 H3 decomposition}. By  Corollary \ref{C:H1 H3 decomposition} we have
	$$
	v_{\sigma,\rho}(H_2)<v_{\sigma,\rho}(H) \quad \text{or}\quad  f_{\sigma,\rho}(H_2)=0.
	$$
	By Lemma \ref{L:New Dixmier 2.4} item 1  we have 
	$$
	v_{\sigma,\rho}(H_1)\leq \max\{v_{\sigma,\rho}(H),v_{\sigma,\rho}(H_2)\}=v_{\sigma,\rho}(H)
	$$ 
	In item 2, we have proved $v_{\sigma,\rho}(H)\leq v_{\sigma,\rho}(L)+v_{\sigma,\rho}(L)$ and equation \ref{E:v sigma rho LM} means $v_{\sigma,\rho}(H_1)=v_{\sigma,\rho}(L)+v_{\sigma,\rho}(M)$. Hence we must have 
	$$
	v_{\sigma,\rho}(H)=v_{\sigma,\rho}(H_1)=v_{\sigma,\rho}(L)+v_{\sigma,\rho}(M),
	$$
	 hence there exists $(l,j)\in E(H)$, such that $l\sigma+j\rho=v_{\sigma,\rho}(H)=v_{\sigma,\rho}(L)+v_{\sigma,\rho}(M)$. By item 1  there exist $(m,u)\in E(L)$ and $(n,v)\in E(M)$, such that $l\leq m+n, j\leq u+v$, thus 
	\begin{equation}
		\label {E:v sigma rho H<}
		v_{\sigma,\rho}(H)=l\sigma+j\rho\leq (m+n)\sigma+(u+v)\rho
	\end{equation}
	But $(m,u)\in E(L)$ and $(n,v)\in E(M)$, this means $\sigma m+\rho u\leq v_{\sigma}(L)$ and $\sigma n+\rho v\leq v_{\sigma}(M)$, hence 
	\begin{equation}
		\label{E:v sigma rho H>}
		v_{\sigma,\rho}(H)=v_{\sigma,\rho}(L)+v_{\sigma,\rho}(M)\ge (m+n)\sigma+(u+v)\rho
	\end{equation}
	Comparing two equations \ref{E:v sigma rho H<} and \ref{E:v sigma rho H>}, we get $m\sigma+u\rho=v_{\sigma,\rho}(L)$ and $n\sigma +v\rho=v_{\sigma,\rho}(M)$, this means $f_{\sigma,\rho}(L)\neq 0$ and $f_{\sigma,\rho}(M)\neq 0$. 
\end{proof}

As a result, we have a way to calculate $f_{\sigma,\rho}(LM)$ only by $f_{\sigma,\rho}(L),f_{\sigma,\rho}(M)$ when $v_{\sigma,\rho}(LM)=v_{\sigma,\rho}(L)+v_{\sigma,\rho}(M)$.

\begin{lemma}
	\label{L: f sigma rho LM}
    In the notations of lemma \ref{L:New Dixmier 2.7}, if $v_{\sigma,\rho}(LM)= v_{\sigma,\rho}(L)+v_{\sigma,\rho}(M)$, then 
	$$
	f_{\sigma,\rho}(LM)=f_{\sigma,\rho}[f_{\sigma,\rho}(L)f_{\sigma,\rho}(M)].
	$$
\end{lemma}

\begin{proof}
	Assume first $f_{\sigma,\rho}(L)\neq 0,f_{\sigma,\rho}(M)\neq 0$. Put $L_1=f_{\sigma,\rho}(L)\neq 0$, $M_1=f_{\sigma,\rho}(M)\neq 0$, put $L_3=L-L_1$ and $M_3=M-M_1$. Consider the equation 
	$$
	H=LM=L_1M_1+L_1M_3+L_3M_1+L_3M_3
	$$
	For $L_3, M_3$, we have 4 possibilities:
	\begin{enumerate}
		\item [(1)] $v_{\sigma,\rho}(L_3)<v_{\sigma,\rho}(L_1)$,$v_{\sigma,\rho}(M_3)<v_{\sigma,\rho}(M_1)$
		\item [(2)] $v_{\sigma,\rho}(L_3)<v_{\sigma,\rho}(L_1)$,$v_{\sigma,\rho}(M_3)=v_{\sigma,\rho}(M_1)$, but $f_{\sigma,\rho}(M_3)=0$.
		\item [(3)] $v_{\sigma,\rho}(L_3)=v_{\sigma,\rho}(L_1)$,$v_{\sigma,\rho}(M_3)<v_{\sigma,\rho}(M_1)$, but $f_{\sigma,\rho}(L_3)=0$.
		\item [(4)] $v_{\sigma,\rho}(L_3)=v_{\sigma,\rho}(L_1)$,$v_{\sigma,\rho}(M_3)=v_{\sigma,\rho}(M_1)$, but $f_{\sigma,\rho}(L_3)=f_{\sigma,\rho}(M_3)=0$.
	\end{enumerate} 
	For (1), we know $v_{\sigma,\rho}(L_1M_3)\leq v_{\sigma,\rho}(L_1)+v_{\sigma,\rho}(M_3)<v_{\sigma,\rho}(L_1)+v_{\sigma,\rho}(M_1)$. By Lemma \ref{L:New Dixmier 2.4} item 2, we have $f_{\sigma,\rho}(L_1M_1+L_1M_3)=f_{\sigma,\rho}(L_1M_1)$, analogously for $L_3M_1$ and $L_3M_3$. We get 
	$$
	f_{\sigma,\rho}(H)=f_{\sigma,\rho}(L_1M_1)
	$$
	For (2), $f_{\sigma,\rho}(M_3)=0$ means for any $(n_3,v_3)\in E(M_3)$  $\sigma n_3+\rho v_3<v_{\sigma,\rho}(M_1)$. We need the following claim:
	
	Claim: There doesn't exist $(l,j)\in E(L_1M_3)$, such that $l\sigma+j\rho\ge  v_{\sigma,\rho}(L_1)+v_{\sigma,\rho}(M_1)$.

	(Proof of the Claim) Assume the converse, then by item 1, there exist $(m_1,u_1)\in E(L_1)$ and $(n_3,v_3)$, such that 
	$l\leq m_1+n_3$ and $j\leq u_1+v_3$, but we know $m_1\sigma+u_1\rho\leq v_{\sigma,\rho}(L_1)$ and $\sigma n_3+\rho v_3<v_{\sigma,\rho}(M_1)$, this is a contradiction.
	
	So this claim shows that $v_{\sigma,\rho}(L_1M_3)<v_{\sigma,\rho}(L)+v_{\sigma,\rho}(M)$ or $v_{\sigma,\rho}(L_1M_3)=v_{\sigma,\rho}(L)+v_{\sigma,\rho}(M)$, but $f_{\sigma,\rho}(L_1M_3)=0$. Like in (1) we can check $v_{\sigma,\rho}(L_3M_1)<v_{\sigma,\rho}(L_1M_1)$, $v_{\sigma,\rho}(L_3M_3)<v_{\sigma,\rho}(L_1M_1)$. So  we get again $f_{\sigma,\rho}(H)=f_{\sigma,\rho}(L_1M_1)$.
	
	Cases (3) and (4) are analogous,  we omit the details.
	
	If at least one of $f_{\sigma,\rho}(L)$ and $f_{\sigma,\rho}(M)=0$, then the above arguments show  there doesn't exist $(l,j)\in E(LM)$ such that $l\sigma+j\rho=v_{\sigma,\rho}(LM)= v_{\sigma,\rho}(L)+v_{\sigma,\rho}(M)$, hence $f_{\sigma,\rho}(LM)=0$.
\end{proof}

\noindent J. Guo,  School of Mathematics and Statistics, Leshan Normal University, Sichuan, China
\\ 
\noindent\ e-mail:
$123281697@qq.com$

\vspace{0.5cm}

\noindent A. Zheglov,  Lomonosov Moscow State  University, Faculty
of Mechanics and Mathematics, Department of differential geometry
and applications and Moscow Center of Fundamental and Applied Mathematics of Lomonosov Moscow State University,  Leninskie gory, GSP, Moscow, \nopagebreak 119899,
Russia
\\ \noindent e-mail
 $azheglov@math.msu.su$, $alexander.zheglov@math.msu.ru$, $abzv24@mail.ru$

\end{document}